\documentclass[letterpaper,article,oneside,11pt]{memoir}

\usepackage{amsmath}
\usepackage{amssymb}
\usepackage{amsthm}
\usepackage{mathrsfs}
\usepackage{lmodern}
\usepackage[T1]{fontenc}
\usepackage[utf8]{inputenc}
\usepackage{enumitem}

\usepackage{tikz-cd}
\usetikzlibrary{cd}
\usepackage{pgfplots}
\pgfplotsset{compat=1.5.1}

\settrims{0pt}{0pt}
\setheadfoot{0.5in}{0.5in}
\setulmarginsandblock{1.0in}{1.0in}{*}
\setlrmarginsandblock{1.0in}{1.0in}{*}
\checkandfixthelayout

\usepackage[backend=biber,style=alphabetic,giveninits,url=false,doi=false,maxbibnames=50,maxcitenames=50,maxnames=50]{biblatex}
\defbibheading{bibliography}[Bibliography]{\section*{#1}\markboth{#1}{#1}}

\counterwithout{section}{chapter}
\counterwithout{equation}{chapter}
\numberwithin{equation}{section}

\renewcommand{\maketitle}{
 {\centering\textbf{\thetitle}\par
  \vspace{0.5em}
  \centering{\theauthor}\par 
  \vspace{2em}
 }
}

\setsecheadstyle{\bfseries}

\setsecnumdepth{subsubsection}
\setsecnumformat{\csname the#1\endcsname. }
\setsubsecheadstyle{\normalfont\bfseries}
\setaftersubsecskip{-0.5em}
\setsubsubsecheadstyle{\normalfont\bfseries}
\setaftersubsubsecskip{-0.5em}

\OnehalfSpacing
\newtheorem{theorem}[equation]{Theorem}

\newtheorem*{claim}{\textit{Claim}}
\newtheorem{proposition}[equation]{Proposition}
\newtheorem{corollary}[equation]{Corollary}
\newtheorem{lemma}[equation]{Lemma}
\newtheorem{fact}[equation]{Fact}

\theoremstyle{definition}

\newtheorem{question}[equation]{Question}

\newcommand{\metric}{\mathsf{d}}
\newcommand{\nbar}{|\!|}
\newcommand{\intd}{\,\mathrm{d}}
\newcommand{\haar}{\mathsf{m}}
\newcommand{\ess}{\mathsf{Ess}}
\newcommand{\nudge}{\mathsf{nudge}}
\newcommand{\coll}{\mathsf{coll}}
\DeclareMathOperator{\symdiff}{\triangle}
\DeclareMathOperator{\lp}{L}
\DeclareMathOperator{\cont}{\mathsf{C}}
\DeclareMathOperator{\contc}{\cont_\mathsf{c}}

\newcommand{\Cesaro}{Ces\`{a}ro}
\newcommand{\define}[1]{\textbf{#1}}
\renewcommand{\emptyset}{\varnothing}

\title{Ergodicity of skew products over linearly recurrent IETs}
\author{Jon Chaika \and Donald Robertson}

\begin{document}

\maketitle

\begin{abstract}
We prove that the skew product over a linearly recurrent interval exchange transformation defined by almost any real-valued, mean-zero linear combination of characteristic functions of intervals is ergodic with respect to Lebesgue measure.
\end{abstract}

\section{Introduction}
\label{sec:introduction}

Let $T$ be an ergodic measure-preserving transformation on a probability space $(X,\mathscr{B},\mu)$.
Given a measurable function $f : X \to \mathbb{R}$ one can consider the \define{skew product} $T_f$ on $X \times \mathbb{R}$ defined by $T_f(x,t) = (Tx,t + f(x))$ for all $x \in X$ and all $t \in \mathbb{R}$.
It is immediate that $T_f$ preserves $\mu \otimes \nu$ where $\nu$ is Lebesgue measure on $\mathbb{R}$.
Atkinson~\cite{MR0419727} proved that $T_f$ is recurrent with respect to $\mu \otimes \nu$ if and only if $f$ has zero mean and Schmidt~\cite[Theorem~5.5]{MR0578731} proved that $T_f$ is conservative if and only if $f$ has zero mean.
It is therefore natural to ask whether $T_f$ is ergodic with respect to $\mu \otimes \nu$ when $f$ has zero mean.

In this paper we are interested in the situation where $T$ is an interval exchange transformation.
We remind the reader that an \define{interval exchange transformation} is specified by a permutation $\pi$ of $\{1,\dots,b\}$ for some $b \in \mathbb{N}$ and by positive lengths $\lambda_1,\dots,\lambda_b$ that sum to 1.
Given such data one defines a map $T : [0,1) \to [0,1)$ by
\[
Tx = x - \sum_{j < i} \lambda_j + \sum_{\pi j < \pi i} \lambda_j
\]
for all $x \in I_i$ where $I_i = [\lambda_0 + \cdots + \lambda_{i-1},\lambda_0 + \cdots + \lambda_i)$ for each $1 \le i \le b$ and $\lambda_0 = 0$.
All interval exchange transformations preserve Lebesgue measure on $[0,1)$.
A permutation $\pi$ on $\{1,\dots,b\}$ is \define{irreducible} if there is not $1 \le k < b$ such that $\pi(\{1,\dots,k\}) = \{1,\dots,k\}$.
Throughout, we only consider interval exchange transformations defined by permutations that are irreducible.

For interval exchange transformations with $b = 2$ (i.e.\ circle rotations) and varying classes of skewing function $f$ the associated skew product $T_f$ was shown to be ergodic by Oren~\cite{MR693356}, Hellekalek and Larcher~\cite{MR853454}, Pask~\cite{MR1046174} and Conze and Pi\k{e}kniewska \cite{MR3265295}.
For special IETs on more intervals Conze and Fr\k{a}czek \cite{MR2770454} proved ergodicity for skew products by certain piecewise linear functions.
Negative results also occur, as Fr\k{a}czek and Ulcigrai \cite{MR3232007} showed typical non-ergodicity for a family of IETs with skewing functions that depend on the intervals of the IETs (which were considered for their relation to certain billiards). 

In this paper we prove that, for linearly recurrent interval exchange transformations (an analogue of badly approximable rotations for interval exchange transformations) the skew product $T_f$ is ergodic for almost every step function $f$ with zero mean.
(By a \define{step function} we mean a linear combination of characteristic function of intervals.)

\begin{theorem}
\label{thm:mainTheorem}
Let $T$ be a linearly recurrent interval exchange transformation.
For almost every mean-zero step function $f : [0,1) \to \mathbb{R}$ the skew product $T_f$ is ergodic.
\end{theorem}

The terms ``linear recurrent'' and ``almost every'' in Theorem~\ref{thm:mainTheorem} require some explanation.
We first recall the definition of linear recurrence for interval exchange transformations.
Let $T$ be an interval exchange transformation and let $\beta_i = \lambda_1 + \cdots + \lambda_i$ for all $1 \le i \le b-1$.
Put $D = \{ \beta_1,\dots, \beta_{b-1} \}$.
One says that $T$ satisfies the \define{infinite distinct orbits condition} if $D \cap (T^n)^{-1} D = \emptyset$ for all $n \in \mathbb{N}$.
Keane~\cite{MR0357739} proved that the infinite distinct orbits condition implies $T$ is minimal in the sense that every point has dense orbit.
An interval exchange transformation $T$ satisfying the infinite distinct orbits condition is said to be \define{linearly recurrent} if
\[
c_3 = \inf \{ n \eta(n) : n \in \mathbb{N} \} > 0
\]
where $\eta(n)$ is the length of the smallest interval in the partition $D \cup \cdots \cup (T^n)^{-1} D$ of $[0,1)$.
Linear recurrence implies the following statement: that there are constants $c_1,c_2 > 0$ such that every finite orbit $x,Tx,\dots,T^{n-1}x$ is $c_1/n$ dense and $c_2/n$ separated.

The condition ``almost every'' in Theorem~\ref{thm:mainTheorem} refers to a particular parameterization of mean-zero step functions we now describe.
Every step function $f : [0,1) \to \mathbb{R}$ with $d > 0$ discontinuities can be written in the form
\begin{equation}
\label{eqn:simpleFunction}
f = y_1 1_{[0,x_1)} + \cdots + y_{d+1} 1_{[x_1 + \cdots + x_d, 1)}
\end{equation}
for some $y_1,\dots,y_{d+1}$ in $\mathbb{R}$ with $y_i \ne y_{i+1}$ for all $1 \le i \le d$ and some $x_1,\dots,x_{d+1}$ in $(0,1)$ that sum to 1.
By the \define{jumps} of any such $f$ we mean the values $y_2 - y_1,\dots,y_{d+1} - y_d$.
The manifold
\[
\mathcal{C}_d
=
\left\{
(x_1,\dots,x_{d+1},y_1,\dots,y_{d+1}) \in (0,1)^{d+1} \times \mathbb{R}^{d+1}
:
\begin{aligned}
x_1 y_1 + \cdots + x_{d+1} y_{d+1} = 0\\
x_1 + \cdots + x_{d+1} = 1\\
y_i \ne y_{i+1} \textrm{ for all } 1 \le i \le d
\end{aligned}
\right\}
\]
parameterizes all mean-zero step functions $f : [0,1) \to \mathbb{R}$ with $d$ discontinuities.
We equip $\mathcal{C}_d$ with the metric $\metric$ induced by the $\ell^\infty$ metric on $\mathbb{R}^{2d+2}$.
In Theorem~\ref{thm:mainTheorem} and throughout the paper, a statement is true for almost every mean-zero step function $f$ if, for every $d \in \mathbb{N}$ the points in $\mathcal{C}_d$ for which the statement is false is a null set for the Lebesgue measure class on $\mathcal{C}_d$.

Our methods also apply to mean-zero step functions $f : [0,1) \to \mathbb{Z}$.
The necessary modifications are to (i) replace $\mathbb{R}^{d+1}$ with $\mathbb{Z}^{d+1}$ in the definition of $\mathcal{C}_d$ and equip $(0,1)^{d+1} \times \mathbb{Z}^{d+1}$ with the natural Lebesgue measure; (ii) redefine nudges (cf.\ Section~\ref{sec:f3f4byNudging} below) to remain $\mathbb{Z}$ valued.
As well as being of intrinsic interest, the resulting skew products are related with $\mathbb{Z}$ covers of compact translation surfaces.
Given such a cover $p : \tilde{M} \to M$ and a direction $\theta$ the first return map on any line segment $\Lambda$ transverse to $\theta$ is an interval exchange transformation $T$ on $\Lambda$.
The first return dynamics on $p^{-1}(\Lambda)$ is equivalent to a skew-product $T_f$ where $f : \Lambda \to \mathbb{Z}$ is a step function.
It follows from \cite{MR2049831} that, for a zero Lebesgue measure but full Hausdorff dimension set of $\theta$ the corresponding interval exchange transformation $T$ is linearly recurrent.
Given such a direction $\theta$ the induced skew product $T_f$ and in turn the flow on $\tilde{M}$ in the direction $\theta$ are both recurrent provided $f$ has mean zero (cf.\ \cite{MR3025751}).

The proof of Theorem~\ref{thm:mainTheorem} is outlined in Section~\ref{sec:skewProducts} and the details are given in the subsequent sections.
We mention here the following questions, in which we are very interested.
\begin{question}
Is Theorem~\ref{thm:linearRecurrenceFriends} true with the assumption of linear recurrence weakened to unique ergodicity (or maybe even just minimality)? 
\end{question}
\begin{question}
Let $f = 1_{[0,\frac 1 2)} - 1_{[\frac{1}{2}, 1)}$.
Is $T_f$ ergodic as a $\mathbb{Z}$-valued skew product for almost every $T$?
\end{question}
\begin{question}
Let $f(x)=\cos(2\pi x)$.
Is $T_f$ ergodic for almost every interval exchange transformation on at least three intervals?
\end{question}

J.\ Chaika is supported in part by NSF grants DMS-135500 and
DMS-1452762, the Sloan foundation and a Warnock chair.
D.\ Robertson is grateful for the support of the NSF via grants DMS-1246989 and DMS-1703597.
This project (in fact a more ambitious one) began in 2005 as a joint project between the first named author and Pascal Hubert. 
We thank Pascal Hubert for many helpful conversations.

\section{Proving ergodicity}
\label{sec:skewProducts}

Fix a minimal interval exchange transformation $T$ on $[0,1)$ and a measurable function $f : [0,1) \to \mathbb{R}$.
As with all skew products $T_f$ is said to be \define{recurrent} if, for every $B \in \mathscr{B}$ with $\mu(B) > 0$ and every $\epsilon > 0$ one has
\[
\mu( B \cap (T^n)^{-1} B \cap \{ x \in [0,1) : T_f^n (x,0) \in [0,1) \times (-\epsilon,\epsilon) \} ) > 0
\]
for some $n \in \mathbb{Z} \setminus \{0\}$.
Atkinson~\cite{MR0419727} proved that $T_f$ is recurrent if and only if $f$ has zero mean.
Since $T$ is minimal it then follows from \cite[Theorem~5.5]{MR0578731} that if $f$ has zero mean then $T_f$ is conservative.
It is therefore reasonable to ask -- assuming $f$ has zero mean -- whether $T_f$ is ergodic with respect to Lebesgue measure $\haar$ on $[0,1) \times \mathbb{R}$.
To answer this question one considers the $\mathbb{R}$ action $V$ defined on $[0,1) \times \mathbb{R}$ by $V^v(x,t) = (x,t+v)$, which preserves $\haar$ and commutes with $T_f$.
Write $\mathscr{Z}_f$ for the $\sigma$-algebra of $T_f$ invariant Borel subsets of $[0,1) \times \mathbb{R}$.
As a consequence of \cite[Theorem~5.2]{MR0578731} and \cite[Corollary~5.4]{MR0578731} the measure $\haar$ is ergodic for $T_f$ if and only if the closed subgroup
\[
\ess(f) = \{ v \in \mathbb{R} : \haar(B \symdiff (V^v)^{-1} B) = 0 \textrm{ for all } B \in \mathscr{Z}_f \}
\]
is all of $\mathbb{R}$.
The members of $\ess(f)$ are the \define{essential values} of $f$.
Our main result is therefore a consequence of the following theorem.

\begin{theorem}
\label{thm:jumpsEssential}
Let $T$ be a linearly recurrent interval exchange transformation.
For almost every mean-zero step function $f : [0,1) \to \mathbb{R}$ each of its jumps is an essential value of $T_f$.
\end{theorem}

\begin{proof}[Proof of Theorem~\ref{thm:mainTheorem} assuming Theorem~\ref{thm:jumpsEssential}]
For almost every mean-zero step function $f$ its jumps generate a dense subgroup of $\mathbb{R}$.
Therefore $\ess(f)$ is dense in $\mathbb{R}$.
\end{proof}

Other works, for example  \cite{MR3265295}, also prove that the jumps of the step function are essential values.
They consider step function skew products over rotations $R$ of the circle, for which one can always find infinitely many times $q_n$ such that
\begin{itemize}
\item
(Denjoy-Koksma) $\left| \displaystyle\sum_{i=0}^{q_n-1} f(R^ix) \right| \leq \mathsf{Var}(f)$
\item
$\underset{n \to \infty}{\lim}d(R^{q_n}x,x)=0$
\end{itemize}
both hold for all $x$.
One then seeks to show there are pairs of level sets of $\sum_{i=0}^{q_n-1}f(R^ix)$ of definite measure where the values of $\sum_{i=0}^{q_n-1}f(R^ix)$ differ by the size of particular jump discontinuities of $f$.
In short, one obtains invariance by looking at sets of definite measure at particular times.
We do not suspect that something like the Denjoy-Koksma inequality holds in our context.
As a substitute, we show that the size of the jumps of the skewing function are essential values by following a pair of nearby points whose values under the skew differ by the size of a jump discontinuity of $f$ for a defnite proportion of an orbit segment.
This approach is outlined below and carried out in Section~\ref{sec:friendsGivesInvariance}.
Such arguments go back at least to Ratner \cite{MR717825}.

In order to prove Theorem~\ref{thm:jumpsEssential} we study for some $B > 0$ the transformation $S_{f,B}$ induced by $T_f$ on the space $X_B := [0,1) \times [-B,B]$.
This is defined almost everywhere because $T_f$ is recurrent.
Normalized Lebesgue measure $\haar_B$ on $X_B$ is $S_{f,B}$ invariant, so $\haar_B$ almost every point $(x,t)$ is generic for an ergodic $S_{f,B}$ invariant probability measure $\mu_{f,B,(x,t)}$ on $X_B$.
The following theorem (proved in Appendix~\ref{apdx:gettingInvariance}) relates vertical invariance of the measures $\mu_{f,B,(x,t)}$ with the essential values of $T_f$.

\begin{theorem}
\label{thm:gettingInvariance}
Let $T$ be an ergodic interval exchange transformation. 
Fix $B > 0$.
Suppose that there is $v \in \mathbb{R}$ such that an $\haar_B$ positive measure set of $(x,t)$ in $X_B$ is generic for an $S_{f,B}$ invariant probability measure $\mu_{f,B,(x,t)}$ that is not singular with respect to $V^v \mu_{f,B,(x,t)}$.
Then every $T_f$ invariant set is $V^v$ invariant.
\end{theorem}

We now describe how Theorem~\ref{thm:gettingInvariance} will be used to prove Theorem~\ref{thm:jumpsEssential}.
Fix an interval exchange transformation $T$.
Given a pair $(x,t) \in [0,1) \times \mathbb{R}$, a mean-zero step function $f : [0,1) \to \mathbb{R}$ and $B > 0$, say that $(x,t)$ and $f$ are \define{right friends} at $B$ if there are constants $\beta > 0$, $\delta > 0$ such that, for every discontinuity $p$ of $f$ there is $K \subset \mathbb{N}$ infinite such that all of the following properties hold for all $k \in \mathbb{N}$.
\begin{enumerate}[label=\textup{\textbf{F\arabic*}}.,ref=\textbf{\textup{F\arabic*}}]
\item\label{fr:1}
For all $0 \le i < 2^{k}$ the transformation $T^i$ is continuous on $[x,x+\frac{2 \delta}{2^k}]$.
\item\label{fr:2}
The family $\{ T^i [x,x+\frac{2\delta}{2^k}] : 0 \le i < 2^{k} \}$ of intervals is pairwise disjoint.
\item\label{fr:3}
There is $0 \le i < 2^{k-1}$ with $p \in T^i[x,x+\frac{\delta}{2^k}]$.
\item\label{fr:4}
No other discontinuity of $f$ belongs to $\displaystyle\bigcup_{i=0}^{2^k-1} T^i [x, x+\tfrac{2\delta}{2^k}]$.
\item\label{fr:5}
$\displaystyle{
\sum_{n=2^{k-1}}^{2^k-1} 1_{[-B,B]} \left( t + \sum_{i=0}^{n-1} f(T^i x) \right)
\ge
\beta \sum_{n=0}^{2^k - 1} 1_{[-B,B]} \left( t + \sum_{i=0}^{n-1} f(T^i x) \right)
}$.
\end{enumerate}
Note that only \ref{fr:5} refers to $t$.
We say that $(x,t)$ and $f$ are \define{left friends} at $B$ if the above is true with all occurrences of $[x,x+\frac{2\delta}{2^k}]$ and $[x,x+\frac{\delta}{2^k}]$ replaced by $[x-\frac{2\delta}{2^k},x]$ and $[x-\frac{\delta}{2^k},x]$ respectively.
Declare $(x,t)$ and $f$ to be \define{friends} at $B$ if they are either left friends at $B$ or right friends at $B$.
Theorem~\ref{thm:jumpsEssential} is now a consequence of the following two results.

\begin{theorem}
\label{thm:friendsGivesInvariance}
Let $T$ be an ergodic interval exchange transformation and let $f$ be a mean-zero step function whose jump discontinuities generate a dense subgroup of $\mathbb{R}$.
If
\begin{equation}
\label{eqn:friendsGivesInvariance}
\haar_B( \{ (x,t) \in X_B : f \textrm{ and } (x,t) \textrm{ are friends at } B \}) \to 1
\end{equation}
as $B \to \infty$ then $T_f$ is ergodic with respect to Lebesgue measure on $[0,1) \times \mathbb{R}$.
\end{theorem}
The proof of Theorem~\ref{thm:friendsGivesInvariance} is given in Section~\ref{sec:friendsGivesInvariance}.
It is proved by showing that for all jumps $v$ of $f$ the measure $\mu_{f,B,(x,t)}$ is not singular with respect to $V^v \mu_{f,B,(x,t)}$ for a positive measure set of $(x,t)$.
By Theorem \ref{thm:gettingInvariance} this implies Theorem \ref{thm:friendsGivesInvariance}. 

The next theorem guarantees that the hypothesis of Theorem~\ref{thm:friendsGivesInvariance} are satisfied when $T$ is a linearly recurrent interval exchange transformation - together with Theorem~\ref{thm:friendsGivesInvariance} it concludes the proof of Theorem~\ref{thm:mainTheorem}.

\begin{theorem}
\label{thm:linearRecurrenceFriends}
Fix a linearly recurrent interval exchange transformation $T$.
For almost every mean-zero step function $f$ we have
\begin{equation}
\label{eqn:friendsGoal}
\haar_B( \{ (x,t) \in X_B : f \textrm{ and } (x,t) \textrm{ are friends at } B \}) \to 1
\end{equation}
as $B \to \infty$.
\end{theorem}

The proof of Theorem~\ref{thm:linearRecurrenceFriends} is based on a density points argument with the following steps.
The details are given in Section~\ref{sec:linearRecurrenceFriends}.
\begin{enumerate}
\item
(Section~\ref{sec:fr1fr2always})
Properties \ref{fr:1} and \ref{fr:2} always hold either on the left or the right.
\item
(Section~\ref{sec:f3f4byNudging})
For any $x,f$ and $k$ we can perturb $f$ to satisfy \ref{fr:3} and \ref{fr:4}.
\item
(Section~\ref{sec:magic})
For all $f$ and almost every $(x,t)$ (with $|t|<B$) there exists infinitely many $k$ satisfying \ref{fr:5}.
\item
(Section~\ref{sec:perturbations})
Small perturbations in $f$ preserve \ref{fr:5}.
\item
(Section~\ref{sec:linearRecurrenceFriendsProof})
The preceding steps imply Theorem~\ref{thm:linearRecurrenceFriends}.
\end{enumerate}


We conclude this section with some preparatory remarks that will be used implicitly in the proofs of the above results.
Firstly, given $d \in \mathbb{N}$ and $D \in \mathbb{N}$ write $\mathcal{C}_{d,D}$ for the set of points in $\mathcal{C}_d$ with $|y_i| \le D$ for all $1 \le i \le d+1$.

\begin{lemma}
\label{lem:measurability}
The map
\begin{equation}
\label{eqn:measurability}
\begin{aligned}
\mathcal{C}_{d,D} \times [0,1) \times \mathbb{R}
&
\to
[0,1) \times \mathbb{R}
\\
(f,(x,t))
&
\mapsto
T_f(x,t)
\end{aligned}
\end{equation}
is measurable, and
\[
(f,(x,t)) \mapsto \phi(T_f^n(x,t))
\]
is measurable for every $n$ in $\mathbb{Z}$ and every measurable function $\phi$ on $[0,1) \times \mathbb{R}$.
\end{lemma}
\begin{proof}
Writing $x_0 = 0$, the map
\begin{align*}
[0,1)^{d+1} \times \mathbb{R}^{d+1} \times [0,1) \times \mathbb{R}
&
\to
[0,1) \times \mathbb{R}
\\
(x_1,\dots,x_{d+1},y_1,\dots,y_{d+1},x,t)
&
\mapsto
\left( Tx, t + \sum_{i=0}^d y_i \cdot 1_{[x_0 + \cdots + x_i,x_0 + \cdots + x_{i+1})}(x) \right)
\end{align*}
is measurable and \eqref{eqn:measurability} is simply its restriction to $\mathcal{C}_{d,D} \times [0,1) \times \mathbb{R}$.
\end{proof}

Define
\[
G_3(B) = \{ (f,(x,t)) \in \mathcal{C}_d \times X_B : (x,t) \textrm{ is generic for an } S_{f,B} \textrm{ invariant probability measure on } X_B \}
\]
for all $B > 0$.
Measurability of the sets $G_3(B)$ follows from Lemma~\ref{lem:measurability}.

Lastly, for each $B > 0$ and each $\epsilon > 0$ note that, by Egoroff's theorem, there is set $G_4(B,\epsilon) \subset X_B$ with $\haar_B$-measure at least $1- \epsilon$ on which the sequence
\[
\frac{1}{N} \sum_{n=0}^{N-1} \delta_{S_{f,B}^n (x,t)}
\]
of measures converges uniformly to $\mu_{f,B,(x,t)}$.
Write
\[
\coll(B,b) = [0,1) \times \Big( [-B,-B+b] \cup [B-b,B] \Big)
\]
for any $B > 0$ and any $b > 0$.
Note that $\coll(B,B) = [0,1) \times [-B,B]$.
For all $b,\tau > 0$ we have
\begin{equation}
\label{eqn:collarChebyshev}
\haar_{B+b}(\{ (x,t) \in X_{B+b} : \mu_{f,B+b,(x,t)}(\coll(B+b,2b)) \ge \tau \}) \le \frac{1}{\tau} \haar_{B+b}(\coll(B+b,2b)) \to 0
\end{equation}
as $B \to \infty$ by Markov's inequality.

In various proofs below we will make use of the following lemma, which we do not prove.

\begin{lemma}
\label{lem:tailsFull}
Let $n \mapsto a_n$ be a sequence that \Cesaro{} converges to $\alpha$.
Fix $\epsilon > 0$ and $0 < \gamma < 1$.
There is $K \in \mathbb{N}$ so large that
\[
\left| \alpha - \frac{1}{N-M} \sum_{n=M}^{N-1} a_n \right| < \epsilon
\]
whenever $N > K$ and $N - M \ge \gamma N$.
\end{lemma}

\section{Proof of Theorem~\ref{thm:friendsGivesInvariance}}
\label{sec:friendsGivesInvariance}

In this section we prove Theorem~\ref{thm:friendsGivesInvariance}.
The argument is a modification of a now standard argument that goes back at least to Ratner~\cite{MR717825}.

\begin{proof}
[Proof of Theorem~\ref{thm:friendsGivesInvariance}]
Fix an ergodic interval exchange transformation $T$ on $[0,1)$ and a mean-zero step function $f : [0,1) \to \mathbb{R}$ with coordinates $(x_1,\dots,x_{d+1},y_1,\dots,y_{d+1})$.
Assume the jump discontinuities $y_{i+1} - y_i$ of $f$ generate a dense subgroup of $\mathbb{R}$.
Fix a discontinuity $p$ of $f$.
Let $v$ be the jump at $p$ and put $b = |v|$.

Fix $0 < \eta < 1/10$.
Choose $B >  \max \{2y_1,\dots,2y_{d+1} \}/(1 - \eta)$ so large that
\begin{gather*}
\haar_{B-b} ( \{ (x,t) \in X_{B-b} : f \textup{ and } (x,t) \textup{ are friends at } B-b \} ) \ge 1 - \eta
\\
\haar_{B}(\{ (x,t) \in X_{B} : \mu_{f,B,(x,t)}(\coll(B,b)) < \eta \}) \ge 1 - \frac{b}{B} - \eta
\\
\haar_{B+b}(\{ (x,t) \in X_{B+b} : \mu_{f,B+b,(x,t)}(\coll(B+b,2b)) < \eta \}) \ge 1 - \frac{2b}{B} - \eta
\\
\haar_{B+b}(\{ (x,t) \in X_{B+b} : \mu_{f,B+b,(x,t)}(\coll(B+b,b)) < \eta/2 \}) \ge 1 - \frac{2b}{B} - \eta
\end{gather*}
all hold.
The first inequality follows from the hypothesis \eqref{eqn:friendsGivesInvariance} while the rest follow from \eqref{eqn:collarChebyshev}.
It follows immediately that the four sets above each have $\haar_{B-b}$ measure at least $1 - \eta$.
Fix $L$ a compact subset of $X_{B-b}$ with $\haar_{B-b}$ measure at least $1 - \eta$ on which $(x,t) \mapsto \mu_{f,B,(x,t)}$ is continuous.
From these choices and $\eta < \frac{1}{10}$ we can find an $\haar_{B-b}$ positive measure set of $(x,t) \in X_{B-b}$ with all the following properties.
\begin{enumerate}
[label=\textbf{\textup{R\arabic*}}.,ref=\textup{\textbf{R\arabic*}}]
\item
\label{rat:1}
$(x,t)$ and $f$ are friends at $B-b$.
\item
\label{rat:3}
$(x,t) \in G_3(B+b) \cap G_3(B)$.
\item
\label{rat:4}
$(x,t)$ belongs to and is a density point of $L \cap G_4(B,\eta) \cap \Big( [0,1) \times \{t\} \Big)$.
\item
\label{rat:5}
$\mu_{f,B+b,(x,t)}(\coll(B+b,2b)) < \eta$.
\item
\label{rat:6}
$\mu_{f,B,(x,t)}(\coll(B,b)) < \eta$.
\item
\label{rat:7}
$\mu_{f,B+b,(x,t)}(X_B) \ge 1 - \frac{\eta}{2}$.
\end{enumerate}
Our goal is to prove that for every $(x,t)$ satisfying \ref{rat:1} through \ref{rat:7} the measures $\mu_{f,B,(x,t)}$ and $V^v \mu_{f,B,(x,t)}$ are not mutually singular.
Indeed if this goal is realized then, since we have a positive measure set of $(x,t) \in X_B$ satisfying \ref{rat:1} through \ref{rat:7}, by Theorem~\ref{thm:gettingInvariance} all $T_f$ invariant sets will be $V^v$ invariant.
In other words $v$ will be an essential value of $f$.
Since $v$ was an arbitrary jump of $f$ and the jumps of $f$ are assumed to generate a dense subgroup of $\mathbb{R}$ we will conclude that $T_f$ is ergodic.
To realize our goal fix $(x,t)$ satisfying \ref{rat:1} through \ref{rat:7}.

\begin{claim}
It suffices to prove that
\begin{equation}
\label{eqn:invarianceGoal}
\int g \intd \mu_{f,B,(x,t)} \le \frac{3}{2 - \eta} \int V^{-v} g \intd \mu_{f,B,(x,t)}
\end{equation}
for all $g \in \contc(X_{B-b})$.
\end{claim}
\begin{proof}[Proof of claim]
If $\mu_1 = \mu_{f,B,(x,t)}$ and $\mu_2 = V^{-v} \mu_{f,B,(x,t)}$ are mutually singular then we can find via \ref{rat:6} disjoint compact sets $H_1 \subset X_B \setminus \coll(B,b)$ and $H_2 \subset X_B$ with $\mu_i(H_i) > 1 - 2\eta$ and $\mu_i(H_{3-i}) = 0$ for each $i \in \{1,2\}$.
There also exist open sets $W_2 \subset X_B$ and $W_1 \subset X_B \setminus \coll(B,b)$ so that $H_i \subset W_i$ and $\mu_i(W_{3-i}) < \eta(1 - 2\eta)$.
We may choose a continuous, non-negative function $0\leq g\leq 1$ such that $1_{H_1} \le g \le 1_{W_1}$.
It is then straightforward that
\[
\int g \intd \mu_2
\le
\mu_2(W_1)
\le
\eta(1 - 2\eta)
\le
\eta \mu_1(H_1)
\le
\eta \int g \intd \mu_1
\le
\frac{3 \eta}{2 - \eta} \int g \intd \mu_2
\]
holds.
But $\eta < 1/2$ so it must be the case that all quantities above are zero, which is impossible.
\end{proof}

To establish \eqref{eqn:invarianceGoal} suppose by \ref{rat:1} that $(x,t)$ and $f$ are right friends at $B-b$.
(The proof when they are left friends is similar and omitted.)
Let $\beta > 0$, $\delta > 0$ be the attendant constants and let $K \subset \mathbb{N}$ be the subset associated with our fixed discontinuity $p$ of $f$.
Let $k_1 < k_2 < \cdots$ be an enumeration of $K$.
By \ref{rat:4} and \ref{fr:3} we can find (because $[x+\frac \delta {2^{k_i}},x+\frac{2\delta}{2^{k_i}}]$ is a definite proportion of $[x,x+\frac{2\delta}{2^{k_i}}]$) for all $i$ large enough, some point $z_i$ in $[x,x + \frac{2\delta}{2^{k_i}}]$ with the following properties: that $(z_i,t)$ belongs to $L \cap G_4(B,\eta)$ and that our discontinuity $p$ is between $T^{\ell_i}(x)$ and $T^{\ell_i}(z_i)$ for some $0 \le \ell_i < 2^{k_i - 1}$.
Using \ref{fr:4} to rule out other discontinuities of $f$ and \ref{fr:2} to rule out a second occurrence of $p$ we have
\begin{align*}
T_f^n(z_i,t)
&
=
\big( T^n(z_i),t + f(z_i) + \cdots + f(T^{n-1} z_i) \big)
\\
&
=
\big( T^n(z_i),t + f(x) + \cdots + f(T^{n-1} x) + f(T^{\ell_i} z_i) - f(T^{\ell_i} x) \big)
\\
&
=
V^v \big( T^n(z_i),t + f(x) + \cdots + f(T^{n-1} x) \big)
\end{align*}
whenever $2^{k_i} \ge n > \ell_i$.
It follows that
\begin{equation}
\label{eqn:invarianceAppears}
\nbar T_f^n(x,t) - V^{-v} (T_f^n(z_i,t)) \nbar_2 = |z_i - x|
\end{equation}
whenever $n > \ell_i$.

To every time $n$ at which $T_f^n(x,t)$ belongs to $X_{B-b}$ corresponds some iterate $r(n)$ of the $S_{f,B}$ orbit of $(x,t)$.
For all such $n$ we also have $T_f^n(z_i,t)$ in $X_B$ and a corresponding iterate $r_i(n)$ of the $S_{f,B}$ orbit of $(z_i,t)$.
That is $S_{f,B}^{r_i(n)}(z_i,t) = T_f^n(z_i,t)$.
Define
\begin{align*}
U_i &= \{ r(n) : \ell_i < n < 2^{k_i} \textup{ and } T_f^n(x,t) \in X_{B-b} \}\\
U_i' &= \{ r_i(n) : \ell_i < n < 2^{k_i} \textup{ and } T_f^n(x,t) \in X_{B-b} \}
\end{align*}
for all large enough $i \in \mathbb{N}$.
Note that
\begin{equation}
\label{eqn:hits}
|U_i| \ge \beta  \sum_{n=0}^{2^{k_i} - 1} 1_{[-B+b,B-b]} \left( t + \sum_{j=0}^{n-1} f(T^j x) \right)
\end{equation}
because $f$ and $(x,t)$ satisfy \ref{fr:5} at $B-b$ via \ref{rat:1}.
By \eqref{eqn:invarianceAppears} we have
\begin{equation}
\label{eqn:invarianceInduced}
\nbar S_{f,B}^{r(n)}(x,t) - V^{-v} (S_{f,B}^{r_i(n)}(z_i,t)) \nbar_2 = |z_i - x|
\end{equation}
whenever $T_f^n(x,t) \in X_{B-b}$ and $n > \ell_i$.

\begin{claim}
We have
\begin{equation}
\label{eqn:xSleeveHits}
|U_i|
\ge
\beta(1-2\eta) \sum_{n=0}^{2^{k_i} - 1} 1_{[-B,B]} \left( t + \sum_{j=0}^{n-1} f(T^j x) \right)
\end{equation}
for all large enough $i$.
\end{claim}
\begin{proof}
First note that
\[
\left| \frac{1}{N} \sum_{n=1}^N 1_{\coll(B,b)}(S_{f,B}^n (x,t)) - \mu_{f,B,(x,t)}(\coll(B,b)) \right| < \eta
\]
if $N$ is large enough by \ref{rat:3}.
Thus
\[
\sum_{n=1}^N 1_{\coll(B,b)}(S_{f,B}^n (x,t)) \le 2 \eta N
\]
by \ref{rat:6}.
In terms of $T_f$ this becomes
\[
\sum_{n=0}^{2^{k_i} - 1} 1_{[-B,-B+b] \cup [B-b,B]} \left( t + \sum_{j=0}^{n-1} f(T^j x) \right)
\le
2\eta \sum_{n=0}^{2^{k_i} - 1} 1_{[-B,B]} \left( t + \sum_{j=0}^{n-1} f(T^j x) \right)
\]
whenever $i$ is large enough.
Combining with \eqref{eqn:hits} gives \eqref{eqn:xSleeveHits}.
\end{proof}

\begin{claim}
We have
\begin{equation}
\label{eqn:ziSleeveHits}
|U_i'|
\ge
\beta(1-2\eta) \sum_{n=0}^{2^{k_i} - 1} 1_{[-B,B]} \left( t + \sum_{j=0}^{n-1} f(T^j z_i) \right)
\end{equation}
for all large enough $i$.
\end{claim}
\begin{proof}[Proof of claim]
Certainly $|U_i| = |U_i'|$.
We have
\[
|U_i|
\ge
\sum_{n=2^{k_i - 1}}^{2^{k_i} -1} 1_{[-B+b,B-b]} \left( t + \sum_{j=0}^{n-1} f(T^j x) \right)
\ge
\beta \sum_{n=0}^{2^{k_i} - 1} 1_{[-B+b,B-b]} \left( t + \sum_{j=0}^{n-1} f(T^j x) \right)
\]
because $\ell_i < 2^{k_i - 1}$ and because $(x,t)$ satisfies \ref{fr:5} at $B-b$ via \ref{rat:1}.
Arguing as in the previous claim, from \ref{rat:3} and \ref{rat:5} we also have
\[
\sum_{n=0}^{2^{k_i} - 1} 1_{[-B-b,-B+b] \cup [B-b,B+b]} \left( t + \sum_{j=0}^{n-1} f(T^j x) \right)
\le
2\eta \sum_{n=0}^{2^{k_i} - 1} 1_{[-B-b,B+b]} \left( t + \sum_{j=0}^{n-1} f(T^j x) \right)
\]
for all $i$ is large enough.
Combining these two estimates gives
\[
|U_i'| \ge \beta (1-2\eta) \sum_{n=0}^{2^{k_i} - 1} 1_{[-B-b,B+b]} \left( t + \sum_{j=0}^{n-1} f(T^j x) \right)
\]
and we conclude that
\[
|U_i'| \ge \beta (1-2\eta) \sum_{n=0}^{2^{k_i} - 1} 1_{[-B,B]} \left( t + \sum_{j=0}^{n-1} f(T^j z_i) \right)
\]
because $T_f^n(x,t) \in X_{B+b}$ whenever $T_f^n(z_i,t) \in X_B$.
\end{proof}

Turning to the validity of \eqref{eqn:invarianceGoal} fix $g \in \contc(X_{B-b})$ and $\epsilon > 0$.
For each time $n$ at which $T_f^n(x,t)$ belongs to $X_B$ let $s(n)$ be the corresponding iterate of $S_{f,B}(x,t)$.
For each time $n$ at which $T_f^n(z_i,t)$ belongs to $X_B$ let $s_i(n)$ be the corresponding iterate of $S_{f,B}(z_i,t)$.
Define
\begin{align*}
W_i &= \{ s(n) : \ell < n < 2^{k_i} \textup{ and } T_f^n(x,t) \in X_B \} \\
V_i &= \{ s_i(n) : \ell_i < n < 2^{k_i} \textup{ and } T_f^n(z_i,t) \in X_B \}
\end{align*}
for all $i \in \mathbb{N}$ large enough, both of which are intervals of natural numbers.
Note that $W_i \supset U_i$ and that $V_i \supset U_i'$.

If we choose $i$ large enough then
\[
\left| \int g \intd \Phi_B(x,t) - \frac{1}{|W_i|} \sum_{n \in W_i} g(S_{f,B}^n(x,t)) \right| < \epsilon
\]
using $(x,t) \in G_3(B)$ from \ref{rat:3} and Lemma~\ref{lem:tailsFull} via \eqref{eqn:xSleeveHits} and $|W_i| \ge |U_i|$.
Since $g$ is supported on $X_{B-b}$ we have
\[
\frac{1}{|W_i|} \sum_{n \in W_i} g(S_{f,B}^n(x,t))
=
\frac{1}{|W_i|} \sum_{n \in U_i} g(S_{f,B}^n(x,t))
\]
for all $i$.
Since every member of $U_i'$ corresponds to a unique member of $U_i$ we have
\[
\left|
\frac{1}{|W_i|} \sum_{n \in U_i} g(S_{f,B}^n(x,t))
-
\frac{1}{|W_i|} \sum_{n \in U_i'} (V^{-v} g) (S_{f,B}^n(z_i,t))
\right|
\le
\epsilon
\]
from \eqref{eqn:invarianceInduced} and uniform continuity of $g$ if $i$ is large enough.
Combining the above three with $V_i \supset U_i'$ and $g \ge 0$ gives
\begin{equation}
\label{eqn:firstStep}
\int g \intd \Phi_B(x,t) \le 2\epsilon + \frac{1}{|W_i|} \sum_{n \in V_i} (V^{-v} g) (S_{f,B}^n(z_i,t))
\end{equation}
for all $i$ large enough.

The point $(z_i,t)$ is generic for $\Phi_B(z_i,t)$ and belongs to $G_4(B,\eta)$ by \ref{rat:4}.
By \eqref{eqn:ziSleeveHits} and $|V_i| \ge |U_i'|$ we can apply Lemma~\ref{lem:tailsFull} to get
\[
\left|
\frac{1}{|V_i|} \sum_{n \in V_i} (V^{-v} g) (S_{f,B}^n(z_i,t))
-
\int V^{-v} g \intd \Phi_B(z_i,t)
\right|
<
\epsilon
\]
if $i$ is large enough.
Since $(x,t) \in L$ by \ref{rat:4} we also have
\[
\left| \int V^{-v} g \intd \Phi_B(z_i,t) - \int V^{-v} g \intd \Phi_B(x,t) \right| < \epsilon
\]
for $i$ large enough.
These inequalities together with \eqref{eqn:firstStep} imply
\[
\int g \intd \Phi_B(x,t)
\le
2 \epsilon \left( 1 + \frac{|V_i|}{|W_i|} \right) + \frac{|V_i|}{|W_i|} \int V^{-v} g \intd \Phi_B(x,t)
\]
for our point $(x,t)$ and our function $g \in \contc(X_{B-b})$.

Finally, for every time $n$ at which $T_f^n(x,t)$ belongs to $X_{B+b}$ let $t(n)$ be the corresponding iterate of the $S_{f,B+b}$ orbit of $(x,t)$.
Put
\[
Y_i = \{ t(n) : \ell_i < n < 2^{k_i} \textup{ and } T_f^n(x,t) \in X_{B+b} \}
\]
for all $i$ large enough and note that $|V_i| \le |Y_i|$.
Since $(x,t)$ is generic for $\Phi_{B+b}(x,t)$ by \ref{rat:3} we have
\[
\lim_{N \to \infty} \frac{|W_i|}{|Y_i|} = \mu_{f,B+b,(x,t)}(X_B) \ge 1 - \frac{\eta}{2}
\]
by \ref{rat:7}.
Choosing $\epsilon$ small enough (depending only on $\eta$) and $i$ large enough gives
\[
\int g \intd \Phi_B(x,t) \le \frac{3}{2} \frac{1}{1 - \frac{\eta}{2}} \int V^{-v} g \intd \Phi_B(x,t)
\]
which is \eqref{eqn:invarianceGoal}.
\end{proof}

\section{Proof of Theorem~\ref{thm:linearRecurrenceFriends}}
\label{sec:linearRecurrenceFriends}

Fix throughout this section a linearly recurrent interval exchange transformation $T$ and attendant constants $c_1$, $c_2$ and $c_3$ as in Section~\ref{sec:introduction}.
Fix $\delta = c_2/4$.

\subsection{Conditions \ref{fr:1} and \ref{fr:2} always hold either on the left or the right}
\label{sec:fr1fr2always}

\begin{lemma}
\label{lem:cond 1}
For all $x\in [0,1)$ and $n\in \mathbb{N}$ at least one of the the following two possibilities hold.
\begin{itemize}
\item
$\{ T^i[x,x+\frac{2\delta}{n}] : 0 \le i < n \}$ consists of $n$ disjoint intervals.
\item
$\{ T^i[x-\frac{2\delta}{n},x] : 0 \le i < n \}$ consists of $n$ disjoint intervals.
\end{itemize}
\end{lemma}
The fact that these are disjoint follows from the next result of  Boshernitzan. (Indeed if $T^iA$ are disjoint sets for $p\leq i\leq q$ then $T^jA$ are disjoint sets for $0\leq j\leq q-p$.)
\begin{lemma}
[{\cite[Lemma 4.4]{MR961737}}]
\label{interval}
If $T$ satisfies the Keane condition and the distance between any discontinuities of $T^{n+1}$ is $s$ then for any interval $J$ with measure at most $ s$ there exist integers ${p \leq 0 \leq q}$ (which depend on $J$) such that
\begin{enumerate}
\item[(1)]
$q-p \geq n$
\item[(2)]
$T^i$ acts continuously on $J$ for $p \leq i \leq q$
\item[(3)]
$T^i(J) \cap T^j(J)= \emptyset$ for $p \leq i < j \leq q$.
\end{enumerate}
\end{lemma}

\begin{proof}
[Proof of Lemma \ref{lem:cond 1}]
Linear recurrence implies the discontinuities of $T^n$ are $\frac{c_2}{n}$ separated.
Writing $\beta_1,\dots,\beta_r$ for the discontinuities of $T$, it follows that $T^i[\beta_j,\beta_j + \frac {c_2}{n}) \cap \{\beta_1,\dots,\beta_r\} = \emptyset$ and $T^i[\beta_j - \frac{c_2}{n},\beta_j) \cap \{\beta_1,\dots,\beta_r\} = \emptyset$ for all $1 \le i \le n$ and all $1 \le j \le r$.
Now consider $T^i (x-\frac{c_2}{n},x+\frac{c_2}{n})$.
If it is not an interval then some discontinuity $\beta$ of $T$ belongs to $T^j (x-\frac{c_2}{n},x+\frac{c_2}{n})$ for some $0 \le j < i$.
If $\beta \in [T^j x - \frac{c_2}{2n},T^j x)$ then by above we have that $T^i [x,x+\frac{c_2}{2n}) \cap \{\beta_1,\dots,\beta_r\} = \emptyset$ for all $0 \le i < n$.
Similarly, if $\beta \in [T^j x, T^j x + \frac{c_2}{2n})$ we have the other possibility.
\end{proof}

\subsection{Conditions \ref{fr:3} and \ref{fr:4} are obtainable by nudging}
\label{sec:f3f4byNudging}

Fix $f$ in $\mathcal{C}_{d,D}$ and $x \in [0,1)$.
Let $(x_1,\dots,x_{d+1},y_1,\dots,y_{d+1})$ be the coordinates of $f$ as in \eqref{eqn:simpleFunction}.
Put $\xi = \min \{x_1,\dots,x_{d+1} \}$.
We verify in this subsection that, by perturbing $f$, we can assume conditions \ref{fr:3} and \ref{fr:4} are true.

We wish to choose $g \in \mathcal{C}_{d}$ close enough to $f$ such that \ref{fr:3} and \ref{fr:4} hold.
We construct $g$ by nudging the locations of the discontinuities of $f$.
Specifically, if we wish to move the location of a discontinuity of $f$ to the left or to the right we adjust the values taken by $f$ on the interval to the right of the discontinuity in such a way that the resulting step function still has zero mean.
Explicitly, to nudge $f$ by moving its $i$th discontinuity from $x_1 + \cdots + x_i$ to $x_1 + \cdots + x_i + \zeta$ we replace $f$ with the step function $\nudge(f,i,\zeta)$ having coordinates
\[
\left(x_1,\dots,x_i + \zeta, x_{i+1} - \zeta,\dots, x_{d+1},y_1,\dots,y_i,\frac{y_{i+1} x_{i+1} - \zeta y_i}{x_{i+1} - \zeta},\dots,y_{d+1} \right) 
\]
which makes sense provided $|\zeta| < \frac{\xi}{2}$.
Recall that $\metric$ denotes the $\ell^\infty$ metric on the data $(x_1,\dots,y_{d+1})$.

\begin{lemma}
\label{lem:nudgeAffectsValues}
If $|\zeta| < \frac{\xi}{2}$ then $\metric(f,\nudge(f,i,\zeta)) < \max \{ |\zeta|, 8 |\zeta| D/3 \xi \}$.
\end{lemma}
\begin{proof}
The values of the skewing function have changed by
\begin{equation}
\label{eqn:nudgeOutputChange}
\left| \frac{y_{i+1} x_{i+1} - \zeta y_i}{x_{i+1} - \zeta} - y_{i+1} \right|
=
\left| \frac{\zeta y_{i+1} - \zeta y_i}{x_{i+1} - \zeta} \right|
\le
\frac{|\zeta| 2 D}{|x_{i+1}- \zeta|}
\le
\frac{|\zeta| 8D}{3 x_{i+1}}
\le
\frac{|\zeta| 8D}{3 \xi}
\end{equation}
in carrying out the nudge.
\end{proof}

\begin{proposition}
\label{prop:closeF3F4}
Fix $f$ in $\mathcal{C}_{d,D}$, $1 \le i \le d$ and $x \in [0,1)$.
For every $k$ in $\mathbb{N}$ with
\begin{equation}
\label{eqn:kLargeNudge}
\frac{c_1 + \delta}{2^{k-1}} < \frac{\xi}{4}
\end{equation}
there is $g$ in $\mathcal{C}_d$ with
\begin{equation}
\label{eqn:closeF3F4}
\metric(f,g)
\le
\frac{2c_1 + 3\delta}{2^{k}} (d+1) \max \left\{ 1, \frac{8D}{3\xi} \right\}
\end{equation}
such that $\metric(g,h) < \frac{\delta}{3 \cdot 2^k}$ implies $x$ and $h$ satisfy \ref{fr:1} through \ref{fr:4} on the left or on the right.
\end{proposition}
\begin{proof}
Fix $k \in \mathbb{N}$ satisfying \eqref{eqn:kLargeNudge}.
Assume that the first possibility in Lemma~\ref{lem:cond 1} is true for $n = 2^k$.
(The alternative is treated similarly.)
There are two cases to consider, according to whether there is a time $0 \le \ell < 2^{k-1}$ at which $T^\ell [x,x+\frac{\delta}{2^k}]$ contains the $i$th discontinuity of $f$.
Note that our assumption on $k$ guarantees that each such interval contains at most one discontinuity of $f$.
 
\textbf{Case 1}:
There is such an $\ell$.
Nudge the discontinuity in $T^\ell[x,x+\frac{\delta}{2^k}]$ by at most $\frac{\delta}{2^k}$ so that it lies in $T^\ell[x+\frac{\delta}{3 \cdot 2^k},x+\frac{2\delta}{3 \cdot 2^k}]$.
For each $0 \le j < 2^k$ with $j \ne \ell$ and $T^j [x,x+\frac{3\delta}{2^k}]$ containing a discontinuity of $f$ we nudge the discontinuity of $f$ by at most 
$\frac{3 \delta}{2^k}$ so that it lies in $T^j[x+\frac{3\delta}{2^k},x+\frac{4 \delta}{2^k}]$.
For the resulting function $g$ we have
\begin{equation}
\label{eqn:distanceForGoodTower}
\metric(f,g) \le \frac{3\delta}{2^k} (d+1) \max \left\{ 1, \frac{8D}{3 \xi} \right\}
\end{equation}
by Lemma~\ref{lem:nudgeAffectsValues}.

\textbf{Case 2}:
There is no such $\ell$.
For each $0 \le j < 2^k$ with $T^j [x,x+\frac{3\delta}{2^k}]$ containing a discontinuity of $f$ we nudge the discontinuity of $f$ by at most $\frac{3\delta}{2^k}$ so that it lies in $T^j[x+\frac{ 3\delta}{ 2^k},x+\frac{4 \delta}{3 \cdot 2^k}]$.
By linear recurrence $\{ T^j x : 0 \le j < 2^{k-1} \}$ is within a distance of at most $\frac{c_1}{2^{k-1}}$ from the $i$th discontinuity of $f$.
We nudge it by at most $\frac{c_1}{2^{k-1}} + \frac{\delta}{2^k}$ to lie within some $T^j[x+\frac{\delta}{3 \cdot 2^k},x+\frac{2 \delta}{3 \cdot 2^k}]$ with $0 \le j < 2^{k-1}$.
For the resulting function $g$ we have
\[
\metric(f,g)
\le
\frac{3\delta + 2c_1}{2^{k}} (d+1) \max \left\{ 1, \frac{8D}{3\xi} \right\}
\]
by Lemma~\ref{lem:nudgeAffectsValues}.

In both cases we have constructed a function $g$ with
\[
\metric(f,g)
\le
\frac{3\delta + 2c_1}{2^{k}} (d+1) \max \left\{ 1, \frac{8D}{3\xi} \right\}
\]
and the following properties:
\begin{itemize}
\item
there is only one discontinuity of $g$ in $\cup \{ T^j [x,x+\frac{\delta}{2^k}] : 0 \le i < 2^k \}$;
\item
there is $0 \le j < 2^{k-1}$ such that $T^j [x,x+\frac{\delta}{2^k}]$ contains the only discontinuity of $g$.
\end{itemize}
Moreover, by the construction every $h$ in $\mathcal{C}_d$ with $\metric(g,h) < \frac{\delta}{3 \cdot 2^k}$ also satisfies the above properties.
Thus $x$ and any such $h$ satisfy \ref{fr:1} through \ref{fr:4} on the right (for this $k$).
\end{proof}


\subsection{Condition \ref{fr:5} holds almost surely}
\label{sec:magic}

Using the material of Appendix~\ref{apdx:quantitativeUniqueErgodicity} we prove the following theorem.

\begin{theorem}
\label{thm:magic}
Fix $D > 0$ and $d \in \mathbb{N}$.
There is $\beta > 0$ such that, for every $f \in \mathcal{C}_{d,D}$ and every $B > 0$ and every $t \in (-B,B)$, almost every $x \in [0,1)$ satisfies
\[
\sum_{n=2^{k-1}}^{2^k-1} 1_{[-B,B]} \left( t + \sum_{i=0}^{n-1} f(T^i x) \right)
\ge
\beta \sum_{n=0}^{2^k - 1} 1_{[-B,B]} \left( t + \sum_{i=0}^{n-1} f(T^i x) \right)
\]
for infinitely many $k$.
\end{theorem}

The first step is to prove the following quantitative version of Atkinson's theorem.

\begin{theorem}
\label{thm:quantitativeAtkinson}
Let $T$ be an aperiodic measure-preserving transformation on a probability space $(X,\mathscr{B},\mu)$ and let $f : X \to \mathbb{R}$ in $\lp^1(X,\mathscr{B},\mu)$ have zero mean.
Further assume that there exists $0 < \gamma < 1$ and $N_0 \in \mathbb{N}$ such that
\begin{equation} \label{eq:assumption}
\left| \sum_{n=0}^{N-1} f(T^n x) \right| < N^\gamma
\end{equation}
for all $x \in X$ and all $N \ge N_0$.
Then for every $\epsilon > 0$ there is $N_1 \in \mathbb{N}$ such that
\[
\sum_{n=0}^{N-1} 1_{(-\epsilon,\epsilon)} \left( \sum_{i=0}^n f(T^i x) \right) > N^{1-\gamma-\epsilon}
\]
for almost every $x$ whenever $N \ge N_1$.
\end{theorem}


\begin{proof}
First, we claim that it suffices to prove for every $\eta > 0$ that there is arbitrarily large $N \in \mathbb{N}$ with
\begin{equation}\label{eq:orbit count}
\left| \left\{ 1 \le L \le N: \sum_{m=0}^N 1_{(-\epsilon,\epsilon)} \left( \, \sum_{n=0}^m f(T^i T^L x) \right) > N^{1-\gamma-\epsilon} \right\} \right|
>
(1-\eta) N
\end{equation}
for all $x$. 
Denote the subset of $\{1,\dots,N\}$ at left by $H_x$ (with the dependence on $N$ implicit).
Let
\[
E_N
=
\left\{ x \in [0,1) : \sum_{n=0}^{N-1} 1_{(-\epsilon,\epsilon)} \left( \sum_{i=0}^n f(T^i x) \right) > N^{1-\gamma-\epsilon} \right\}
\]
and notice that if $j \in H_x$ then $T^jx \in E_N$.  Now, if we have \eqref{eq:orbit count} then 
\[
\int_X \sum_{i=1}^N 1_{E_N}(T^ix) \intd \mu \geq (1-\eta)N
\]
and so $\mu(E_N)\geq \frac{(1-\eta)N}N$.

Now we prove \eqref{eq:orbit count}.
Fix $\eta > 0$.
If $N$ is large enough then $f(x) + \cdots + f(T^{N-1} x)$ belongs to $[-N^\gamma,N^\gamma]$ for all $1 \le L \le N$ and all $x \in X$.
Fix an interval $J \subset [-N^\gamma,N^\gamma]$ of length $\epsilon$.
Let $0 \le L_1 < \cdots < L_s \le N$ be an enumeration of those $1 \le L \le N$ at which $f(x) + \cdots + f(T^{L-1} x)$ belongs to $J$.
We have
\[
\epsilon
>
\left| \sum_{n=0}^{L_j - 1} f(T^n x) - \sum_{n=0}^{L_i - 1} f(T^n x) \right|
=
\left| \sum_{n=0}^{L_j - L_i - 1} f(T^n T^{L_i} x) \right|
\]
for all $0 \le i < j \le s$.
Therefore $L_i$ belongs to $H_x$ whenever $s - i \ge N^{1 - \gamma - \epsilon}$ holds.
Since we can cover $[-N^\gamma,N^\gamma]$ by at most $\lceil 2N^\gamma \epsilon^{-1} \rceil$ intervals of length $\epsilon$, it follows that at most $\lceil 2N^\gamma \epsilon^{-1} \rceil N^{1 - \gamma - \epsilon}$ of the $1 \le L \le N$ do not belong to $H_x$.
For $N$ large enough (and independent of $x$) we will therefore have $|H_x| > (1-\eta) N$.
%
\end{proof}

\begin{lemma}\label{lem:poly many}
Fix $0 < \alpha < 1$ and suppose $g : \mathbb{N} \to \mathbb{N}$ is non-decreasing and satisfies $n^\alpha \le g(n) \le n$ for infinitely many $n \in \mathbb{N}$.
Then for every $0 < \beta < 1 - 2^{-\alpha}$ one has $g(2^k) - g(2^{k-1}) \ge \beta g(2^k)$ infinitely often.
\end{lemma}
\begin{proof}
Fix $0 < \beta < 1 - 2^{-\alpha}$.
Suppose the conclusion is false.
Then there is $K \in \mathbb{N}$ such that $(1-\beta) g(2^k) < g(2^{k-1})$ for all $k \ge K$.
Thus $(1 - \beta)^l g(2^{k+l}) < g(2^k)$ for all $l \in \mathbb{N}$ and all $k \ge K$ by induction.
Write $n_j$ for the increasing sequence of times $n \ge 2^K$ at which $n^\alpha \le g(n) \le n$ holds.
For each $j$ fix $l_j \in \mathbb{N} \cup \{0\}$ with $2^{K+l_j} \le n_j < 2^{K+l_j + 1}$.
We have
\[
(1-\beta)^{-(1 + \ell_j)} g(2^K) \geq g(2^{K + l_j + 1}) \ge g(n_j) \ge n_j^\alpha \ge 2^{(K+l_j)\alpha}
\]
for all $j \in \mathbb{N}$.
But then
\[
\Big( (1-\beta) 2^\alpha \Big)^{l_j}
\le
\frac{2^K}{(1-\beta) 2^{K \alpha}}
\]
for all $j \in \mathbb{N}$.
Taking $j$ large enough gives the desired contradiction because $(1 - \beta)2^\alpha > 1$ and so the left hand side goes to infinity with $j$ while the right hand side is independent of $j$.
\end{proof}

\begin{proof}[Proof of Theorem \ref{thm:magic}]
Let $B$ and $t$ be given with $-B<t<B$. It suffices to show that for almost every $x$ we have 
that there exists infinitely many $k$ so that  
\[
\sum_{n=2^{k-1}}^{2^k-1} 1_{[-B-t,B-t]} \left(\; \sum_{i=0}^{n-1} f(T^i x) \right)
\ge
\beta \sum_{n=0}^{2^k - 1} 1_{[-B-t,B-t]} \left(\; \sum_{i=0}^{n-1} f(T^i x) \right).
\] 
By Lemma \ref{lem:poly many} it suffices to show that there exists $\gamma>0$ and an infinite sequence of $N_j$ so that 
\[
\sum_{i=0}^{N_j} 1_{[-B-t,B-t]}\left(\; \sum_{i=0}^{n-1} f(T^i x) \right)
>
(N_j)^\gamma
\]
holds.
Letting $0<\epsilon<|B|-|t|$ and invoking Theorem \ref{thm:quantitativeAtkinson} (which we may do because Theorem \ref{thm:powerSaving} shows \eqref{eq:assumption} is satisfied) gives this condition.
Moreover, since Theorem \ref{thm:powerSaving} produces $0 < \gamma < 1$ uniform over all step functions $f$ we have that $\beta$ is uniform over those functions as well.
\end{proof}

\subsection{Condition \ref{fr:5} survives perturbations}
\label{sec:perturbations}

Here we prove that if \ref{fr:5} holds for a specific pair $(f,(x,t))$ then it also holds if $f$ is perturbed a little.


\begin{proposition}
\label{prop:collarMagic}
Fix $D > 0$ and $d \in \mathbb{N}$.
Let $\beta > 0$ be as in Theorem~\ref{thm:magic}.
Fix $B > C > 0$ and $\tau,\theta > 0$.
Suppose given $f \in \mathcal{C}_{d,D}$ and $(x,t) \in X_{B+C}$ and $K \subset \mathbb{N}$ infinite such that
\begin{enumerate}
[label=\textbf{\textup{C\arabic*}}.,ref=\textbf{\textup{C\arabic*}}]
\item
\label{cm:magic}
\ref{fr:5} holds with $B+C$ in place of $B$ for all $k \in K$;
\item
\label{cm:generic}
$(f,(x,t)) \in G_3(B+C)$;
\item
\label{cm:collar}
$\mu_{f,(x,t),B+C} \Big( \coll(B+C,2C) \Big) < \dfrac{1 - \tau}{1+\theta}$;
\end{enumerate}
all hold.
Then there is $K' \subset K$ cofinite in $K$ such that
\begin{equation}
\label{eqn:magicCollarGoal}
\sum_{n=2^{k-1}}^{2^k-1} 1_{[-B+C,B-C]} \left( t + \sum_{i=0}^{n-1} f(T^i x) \right)
\ge
\beta \tau
\sum_{n=0}^{2^k-1} 1_{[-B-C,B+C]} \left( t + \sum_{i=0}^{n-1} f(T^i x) \right)
\end{equation}
for all $k \in K'$.
\end{proposition}
\begin{proof}
Let $f$ and $(x,t)$ be as in the hypothesis.
By \ref{cm:collar} we can fix $\eta > 0$ such that
\begin{equation}
\label{eqn:magicCollarGap}
(1 + \theta) \mu_{f,(x,t),B+C} \Big( \coll(B+C,2C) \Big) + \eta < 1 - \tau
\end{equation}
holds.
By \ref{cm:magic} we have
\begin{equation}
\label{eqn:bigCollarMagic}
\sum_{n=2^{k-1}}^{2^k-1} 1_{[-B-C,B+C]} \left(t+ \sum_{i=0}^{n-1} f(T^i x) \right)
\ge
\beta \sum_{n=0}^{2^k-1} 1_{[-B-C,B+C]} \left(t+ \sum_{i=0}^{n-1} f(T^i x) \right)
\end{equation}
for all $k \in K$.
Now $(f,(x,t)) \in G_3(B+C)$ by \ref{cm:generic}, so we can apply Lemma~\ref{lem:tailsFull} to get $L \in \mathbb{N}$ so large that
\[
\frac{1}{N-M} \sum_{n=M}^{N-1} 1_{\coll(B+C,2C)} \left( S_{f,B+C}^n (x,t) \right)
\le
(1 + \theta) \mu_{f,(x,t),B+C} \Big( \coll(B+C,2C) \Big) + \eta
\]
whenever $N > L$ and $N-M \ge \beta N$.
There is $K' \subset K$ cofinite such that
\[
N := \sum_{n=0}^{2^k-1} 1_{[-B-C,B+C]} \left( t + \sum_{i=0}^{n-1} f(T^i x) \right)
\]
is at least $L$ whenever $k \in K'$ and, choosing
\[
M := \sum_{n=0}^{2^{k-1}-1} 1_{[-B-C,B+C]} \left( t + \sum_{i=0}^{n-1} f(T^i x) \right)
\]
it follows from \eqref{eqn:bigCollarMagic} that $N-M \ge \beta N$.
Therefore
\begin{align*}
&
\sum_{n=2^{k-1}}^{2^k-1} 1_{[-B-C,-B+C] \cup [B-C,B+C]} \left( t + \sum_{i=0}^{n-1} f(T^i x) \right)
\\
\le
&
\left( (1 + \theta) \mu_{f,(x,t),B+C} \Big( \coll(B+C,2C) \Big) + \eta \right) \sum_{n=2^{k-1}}^{2^k-1} 1_{[-B-C,B+C]} \left( t + \sum_{i=0}^{n-1} f(T^i x) \right)
\end{align*}
for all $k \in K'$ because times $n$ at which $t + \displaystyle\sum_{i=0}^{n-1} f(T^i x)$ belongs to $[-B-C,-B+C] \cup [B-C,B+C]$ are in bijective correspondence with the visits of the $S_{f,B+C}$ orbit of $(x,t)$ to $\coll(B+C,2C)$.
We therefore have
\begin{align*}
&
\sum_{n=2^{k-1}}^{2^k - 1} 1_{[-B+C,B-C]} \left( t + \sum_{i=0}^{n-1} f(T^i x) \right)
\\
\ge \;
&
\left( 1 - (1+\theta) \mu_{f,(x,t),B+C} \Big( \coll(B+C,2C) \Big) - \eta \right) \sum_{n=2^{k-1}}^{2^k-1} 1_{[-B-C,B+C]} \left( t + \sum_{i=0}^{n-1} f(T^i x) \right)
\end{align*}
for all $k \in K'$.
Using \eqref{eqn:magicCollarGap} and \eqref{eqn:bigCollarMagic} gives \eqref{eqn:magicCollarGoal} immediately.
\end{proof}

Recall that if $f$ in $\mathcal{C}_d$ has coordinates $(x_1,\dots,x_{d+1},y_1,\dots,y_{d+1})$ then $\xi$ denotes $\min \{ x_1,\dots,x_{d+1} \}$.

\begin{theorem}
\label{thm:perturbPoint}
Fix $D > 0$ and $d \in \mathbb{N}$ and $A > 0$.
Let $\beta > 0$ be as in Theorem~\ref{thm:magic}
Put
\begin{equation}
\label{eqn:perturbMargin}
C = 6 d D A \frac{c_1}{c_2} + 2dD + 4dAc_1
\end{equation}
and fix $B > C$.
Fix $\tau > 0$.
Given $f \in \mathcal{C}_{d,D}$ and $(x,t)$ and $K \subset \mathbb{N}$ satisfying \ref{cm:magic}, \ref{cm:generic} and \ref{cm:collar} we can find $K' \subset K$ cofinite such that, whenever $k \in K'$ and $g \in \mathcal{C}_{d,D}$ satisfy
\begin{equation}
\label{eqn:funcPerturbReq}
\metric(f,g)
\le
\min \left\{ \frac{\xi}{4(d+1)}, \frac{A(c_1 + \delta)}{2^k} \right\}
\end{equation}
then
\[
\sum_{n=2^{k-1}}^{2^k-1} 1_{[-B,B]} \left( t + \sum_{i=0}^{n-1} g(T^i x) \right)
\ge
\beta \tau \sum_{n=0}^{2^k-1} 1_{[-B,B]} \left( t + \sum_{i=0}^{n-1} g(T^i x) \right)
\]
holds.
\end{theorem}
\begin{proof}
Let $f \in \mathcal{C}_{d,D}$ and $(x,t)$ and $K \subset \mathbb{N}$ infinite satisfy \ref{cm:magic}, \ref{cm:generic} and \ref{cm:collar}.
Let $K'$ be as in the conclusion of Proposition~\ref{prop:collarMagic}.
Fix $k \in K'$.
Fix $g$ with coordinates $(\tilde{x}_1,\dots,\tilde{x}_{d+1},\tilde{y}_1,\dots,\tilde{y}_{d+1})$ satisfying \eqref{eqn:funcPerturbReq}.

Let $I_j$ be the interval with endpoints $x_1 + \cdots + x_j$ and $\tilde{x}_1 + \cdots + \tilde{x}_j$ for each $1 \le j \le d+1$.
These intervals are disjoint by \eqref{eqn:funcPerturbReq}.
We have $|f(x) - g(x)| \le 2D$ on each such interval.
Off these intervals we have $|f(x) - g(x)| \le \metric(f,g)$.
Put $J_j = \{ 0 \le i < 2^k : T^i x \in I_j \}$ for each $1 \le j \le d+1$ and $J = J_1 \cup \cdots \cup J_{d+1}$.
The complement of the intervals $I_i$ is a collection $I'_1,\dots,I'_{d+1}$ of $d+1$ disjoint intervals with respective widths at most $x_i$.
Linear recurrence implies
\[
|J_j| \le\lceil 2^k |I_j| /c_2 \rceil\le \lceil 2^k \metric(f,g)/c_2 \rceil
\]
for all $j$ and that the orbit $x,\dots,T^{2^k-1} x$ is in $I'_i$ at most $\lceil 2^k x_i /c_2 \rceil$ times.
Now for each $0 < n \le 2^k$ we estimate that
\begin{align*}
\left| \sum_{i=0}^{n-1} f(T^i x) - \sum_{i=0}^{n-1} g(T^i x) \right|
\le\,
&
\sum_{i \in J} |f(T^i x) - g(T^i x) | + \sum_{i \notin J} |f(T^i x) - g(T^i x) |
\\
\le\,
&
d \left( \frac{ 2^k \metric(f,g)}{c_2} + 1 \right) 2D + \sum_{i=1}^{d+1} \left( \frac{2^k x_i}{c_2} + 1 \right) \metric(f,g)
\\
\le\,
&
2dD \frac{A(c_1 + \delta)}{c_2} + 2dD + \frac{A(c_1 + \delta)}{c_2} + (d+1) A(c_1 +\delta)
\\
\le\,
&
4dDA \frac{c_1}{c_2} + 2dD + 2A \frac{c_1}{c_2} + 4dA c_1
\\
\le\,
&
6 d D A \frac{c_1}{c_2} + 2dD + 4dAc_1
\end{align*}
using \eqref{eqn:funcPerturbReq} and $c_1 \ge c_2$ and $\delta = c_2/4$.
This gives the implications
\begin{align*}
t+\sum_{i=0}^{n-1} f(T^i x) \in [-B+C,B-C]
\Rightarrow\,
&
t+\sum_{i=0}^{n-1} g(T^i x) \in [-B,B]
\\
t+\sum_{i=0}^{n-1} g(T^i x) \in [-B,B]
\Rightarrow\,
&
t+\sum_{i=0}^{n-1} f(T^i x) \in [-B-C,B+C]
\end{align*}
for all $0 < n \le 2^k$.
Combining with \eqref{eqn:magicCollarGoal} we get
\begin{align*}
\sum_{n=2^{k-1}}^{2^k-1} 1_{[-B,B]} \left(t+ \sum_{i=0}^{n-1} g(T^i x) \right)
\ge
&
\sum_{n=2^{k-1}}^{2^k-1} 1_{[-B+C,B-C]} \left( t + \sum_{i=0}^{n-1} f(T^i x) \right)
\\
\ge
&
\beta \tau \sum_{n=0}^{2^k-1} 1_{[-B-C,B+C]} \left( t + \sum_{i=0}^{n-1} f(T^i x) \right)
\\
\ge
&
\beta \tau \sum_{n=0}^{2^k - 1} 1_{[-B,B]} \left( \sum_{i=0}^{n-1} t + g(T^i x) \right)
\end{align*}
for all $k \in K'$ as desired.
\end{proof}

\subsection{Conditions \ref{fr:1} though \ref{fr:5} hold almost surely}
\label{sec:linearRecurrenceFriendsProof}

Fix throughout this subsection $d \in \mathbb{N}$ and $D \in \mathbb{N}$.
Let $\beta$ be as in Theorem~\ref{thm:magic}.
Given a mean-zero step function $f$ in $\mathcal{C}_{d,D}$ put
\begin{equation}
\label{eqn:bigA}
A = 2 (d+1) \frac{10 \delta + 6 c_1}{3 \delta + 3 c_1} \max \left \{ 1, \frac{8D}{3\xi} \right\}
\end{equation}
where $\xi = \min \{ x_1,\dots,x_{d+1} \}$ and write $C = 6 d D A c_1/c_2 + 2dD + 4dAc_1$ as in \eqref{eqn:perturbMargin}.

\begin{lemma}
\label{lem:densityProps}
Fix $\tau > 0$.
For every $\epsilon > 0$ and every $f \in \mathcal{C}_{d,D}$ there is $B > 0$ such that the $\haar_{B+C}$ measure of the set of pairs $(x,t) \in X_{B+C}$ for which the following statement is true is at least $1 - \epsilon$: there are $\theta > 0$ and $K \subset \mathbb{N}$ infinite such that the following conditions all hold.
\begin{enumerate}
[label=\textup{\textbf{L\arabic*}}.,ref=\textup{\textbf{L\arabic*}}]
\item
\label{ld:1}
$(f,(x,t)) \in G_3(B+C)$.
\item
\label{ld:2}
$\mu_{f,B+C,(x,t)} \Big( \coll(B+C,2C) \Big) < \dfrac{1 - \tau}{1+\theta}$.
\item
\label{ld:3}
\ref{fr:5} with $B+C$ in place of $B$ for all $k \in K$.
\item
\label{ld:4}
$\dfrac{A(c_1 + \delta)}{2^k} \le \dfrac{\xi}{4(d+1)}$ for all $k \in K$.
\item
\label{ld:5}
$\dfrac{c_1 + \delta}{2^{k-1}} < \dfrac{\xi}{4}$ for all $k \in K$.
\end{enumerate}
\end{lemma}
\begin{proof}
Fix $\tau > 0$, $\epsilon > 0$ and $f \in \mathcal{C}_{d,D}$.
For every $B > 0$ almost every $(x,t) \in X_{B+C}$ has the property that it is generic for an $S_{f,B}$ invariant probability measure $\mu_{f,B+C,(x,t)}$ so \ref{ld:1} holds almost surely for all $B$.
For \ref{ld:2} note that
\[
\haar_{B+C}\left( \left\{ (x,t) \in X_{B+C} : \mu_{f,B+C,(x,t)}(\coll(B+C,2C)) \ge \frac{1 - \tau}{1 + \theta} \right\} \right)
\le
\frac{1 + \theta}{1-\tau} \haar_{B+C}(\coll(B+C,2C))
\]
by Markov's inequality and that the right-hand side is less than $\epsilon$ for $B$ large enough independent of $(x,t)$.
Theorem~\ref{thm:magic} implies \ref{ld:3} is true for almost all $(x,t) \in X_{B+C}$.
By removing finitely many points from $K$ we get \ref{ld:4} and \ref{ld:5}.
\end{proof}

\begin{proof}[Proof of Theorem~\ref{thm:linearRecurrenceFriends}]
Fix $0 < \tau < 1$.
For each $(x,t) \in X_B$ and every $q \in \mathbb{N}$ we can consider the set
\[
E_B(q,(x,t)) = \bigcap_{k \ge q} \{ e \in \mathcal{C}_{d,D} : \textup{one of } \ref{fr:3},\ref{fr:4},\ref{fr:5} \textup{ at } \beta\tau \textup{, fails for } (x,t) \textup{ and } e \}
\]
of skewing functions $e$ which fail to be friends with $(x,t)$ because there are no $k$ larger than $q$ for which all of \ref{fr:3}, \ref{fr:4}, \ref{fr:5} at $\beta \tau$ hold.
Fix $f \in \mathcal{C}_{d,D}$.

\begin{claim}
If $f$ and $(x,t)$ satisfy \ref{ld:1} through \ref{ld:5} for some $\theta > 0$ and some $K \subset \mathbb{N}$ infinite then $f$ is not a density point of $E_B(q,(x,t))$.
\end{claim}
\begin{proof}[Proof of claim]
Fix $f$ and $(x,t)$ such that \ref{ld:1} through \ref{ld:5} are satisfied for some $\theta > 0$ and some $K \subset \mathbb{N}$ infinite.
By \ref{ld:1}, \ref{ld:2} and \ref{ld:3} we can apply Theorem~\ref{thm:perturbPoint}, by which there is $K' \subset K$ cofinite with the following property: if $h \in \mathcal{C}_{d,D}$ satisfies
\begin{equation}
\label{eqn:closeMagic}
\metric(f,h) \le \min \left\{ \frac{\xi}{4(d+1)}, \frac{A(c_1 + \delta)}{2^k} \right\}
\end{equation}
for some $k \in K'$ then
\begin{equation}
\label{eqn:useOfNudging}
\sum_{n=2^{k-1}}^{2^k-1} 1_{[-B,B]} \left(t + \sum_{i=0}^{n-1} h(T^i x) \right)
\ge
\beta \tau \sum_{n=0}^{2^k - 1} 1_{[-B,B]} \left(t + \sum_{i=0}^{n-1} h(T^i x) \right)
\end{equation}
holds.

Fix $k \in K'$.
By \ref{ld:5} we can apply Proposition~\ref{prop:closeF3F4} to get $g \in \mathcal{C}_d$ with the properties therein.
That is, for any $h$ in $\mathcal{C}_d$ with $\metric(g,h) < \frac{\delta}{3 \cdot 2^k}$ we have that $h$ and $(x,t)$ satisfy \ref{fr:1} through \ref{fr:4} on either the left or the right for our current value of $k$.
Now
\begin{align*}
\metric(f,h)
&
\le
\frac{2c_1 + 3 \delta}{2^k} (d+1) \max \left\{ 1, \frac{8D}{3\xi} \right\} + \frac{\delta}{3 \cdot 2^k}
\\
&
\le
\frac{A(c_1 + \delta)}{2^k}
\\
&
\le
\min \left\{ \frac{\xi}{4(d+1)}, \frac{A(c_1 + \delta)}{2^k} \right\}
\end{align*}
for any such $h$ upon using \eqref{eqn:closeF3F4}, \eqref{eqn:bigA} and \ref{ld:4}.
By the previous paragraph, this implies \eqref{eqn:useOfNudging} holds.
So, for our current value of $k \in K'$ the pair $(x,t)$ and the function $h$ satisfy \ref{fr:1} through \ref{fr:5} either on the left or the right with $\tau \beta$ in place of $\beta$.

To summarize, for each $k \in K'$ every $h$ in the ball (with respect to the $\metric$ metric) centered at $g$ of radius
\[
r(k) = \frac{\delta}{3 \cdot 2^k}
\]
satisfies \ref{fr:1} through \ref{fr:5} with $\tau \beta$ in place of $\beta$ either on the left or on the right.
This ball is entirely contained within the ball centered at $f$ of radius
\[
R(k) = r(k) + \frac{2c_1 + 3 \delta}{2^k} (d+1) \max \left\{ 1, \frac{8D}{3\xi} \right\}
\]
and $\inf \{ R(k) / r(k) : k \in K' \} > 0$ so taking $k > q$ we conclude that $f$ is not a density point for the set $E_B(q,(x,t))$.
\end{proof}

Combining the claim with Lemma~\ref{lem:densityProps} gives, for almost every $f \in \mathcal{C}_{d,D}$, that
\[
\haar_B(\{ (x,t) \in X_B : f \textup{ not a density point of } E_B(q,(x,t)) \}) \to 1
\]
as $B \to \infty$.
Moreover, the convergence is uniform in $q$ because (as in Lemma~\ref{lem:densityProps}) we need only take $B$ large enough that $\haar_{B+C}(\coll(B+C,2C))$ is small.
Fix now any probability $\eta$ on $\mathcal{C}_{d,D}$ that is equivalent to Lebesgue measure in every atlas.
For each $B$ and every $(x,t) \in X_B$ the sequence $q \mapsto E_B(q,(x,t))$ of sets is increasing, so
\[
\eta \Big( \bigcup_{q \in \mathbb{N}} E_B(q,(x,t)) \Big)
\le
\limsup_{q \to \infty} \eta ( \{ \textup{Density points of } E_B(q,(x,t)) \} )
\]
holds.
But
\[
\haar_B(\{ (x,t) \in X_B : \eta ( \{ \textup{Density points of } E_B(q,(x,t)) \} ) > \epsilon \}) \to 0
\]
as $B \to \infty$ so we conclude the set
\[
\haar_B \left( \left\{ (x,t) \in X_B : \eta \left( \bigcap_{B=1}^\infty \bigcup_{q \ge 1} E_B(q,(x,t)) \right) = 0 \right\} \right)
\to
0
\]
as $B \to \infty$.
By Fubini's theorem, for almost every $f$ we have \eqref{eqn:friendsGoal}.
\end{proof}

\appendix

\renewcommand{\thesection}{\Alph{section}}

\section{An ergodic decomposition}
\label{apdx:gettingInvariance}

In this appendix we prove Theorem~\ref{thm:gettingInvariance}.
We will do this using an ergodic decomposition result for quasi-invariant measures due to Schmidt that we reproduce here for convenience.
For measures $\mu,\nu$ on a measure space $(X,\mathscr{B})$ recall that $\mu$ and $\nu$ are \define{equivalent} written $\mu \sim \nu$ when each is absolutely continuous with respect to the other.
Also $\mu,\nu$ are \define{mutually singular} written $\mu \perp \nu$ if there is a measurable set $A \subset X$ with $\mu(A^\mathsf{c}) = 0 = \nu(A)$.

\begin{theorem}
[{\cite[Theorems 6.6 and 6.7]{MR0578731}}]
\label{thm:decompositionSchmidt}
Let $(X,\mathscr{B})$ be a measurable space and let $F : X \to X$ be a $\mathcal{B}$ measurable map.
Fix an $F$ quasi-invariant probability measure $\mu$ on $(X,\mathscr{B})$.
There exists a standard Borel space $(Y,\mathscr{Y})$, a surjective, measurable map $\psi: X \to Y$, and a family $y \mapsto q_y$ of Borel probability measures on $(X,\mathscr{B})$ with the following properties.
\begin{enumerate}
\item \label{conc:borel}
The map $y \mapsto q_y(A)$ is Borel for every $A \in \mathscr{B}$.
\item
Against all members of $\mathscr{B}$ one has $\mu = \displaystyle\int q_y \intd \rho(y)$ where $\rho = \psi \mu$. 
\item
All of the measures $q_y$ are $F$ quasi-invariant and ergodic. 
\item \label{conc:all points}
$q_y(\psi^{-1}(y))=1$ for every $y\in Y$.
\item
\label{conc:invarAlg}
Let $\mathscr{Z}$ be the $\sigma$-algebra of $F$ invariant sets.
Let $\mathscr{C} = \{ \psi^{-1}(B) : B \in \mathscr{Y}\}$.
Then $\mathscr{Z}$ and $\mathscr{Z}$ are $\mu$ equivalent.
\item
\label{conc:map}
If there is another collection $(Y',\mathscr{Y}',\psi',q')$ satisfying 1.\ through 4.\ then there exists a measurable map $\Theta : Y \to Y'$ that is an isomorphism between $(Y,\mathscr{Y},\rho)$ and $(Y',\mathscr{Y}',\rho')$ such that $q'_{\Theta(y)}\sim q_y$ for $\rho$ almost every $y \in Y$.
\end{enumerate}
\end{theorem}
The following corollary is virtually the same as \cite[Theorem~6.9]{MR0578731}.
We reformulate it slightly for our purposes.

\begin{corollary}[cf.\ {\cite[Theorem~6.9]{MR0578731}}]
\label{cor:decompositionSchmidt}
Let $(X,\mathscr{B})$ be a measurable space and let $F : X \to X$ be a $\mathscr{B}$ measurable map.
Fix an $F$ invariant $\sigma$-finite measure $\nu$ on $(X,\mathscr{B})$.
There exists a Borel space $(Y,\mathscr{Y})$, a surjective Borel measurable map $\psi : X \to Y$, and a family $y \mapsto p_y$  of $\sigma$-finite Borel measures on $(X,\mathscr{B})$ with the following properties.
\begin{enumerate}
[label=\textup{\textbf{S\arabic*}}.,ref=\textup{\textbf{S\arabic*}}]
\item
\label{sch:mble}
The map $y \to p_y(A)$ is measurable for every $A \in \mathscr{B}$.
\item
\label{sch:disint}
Against all members of $\mathscr{B}$ one has $\nu = \displaystyle\int p_y d\rho(y)$ where $\rho=\psi\nu$.
\item
\label{sch:erg}
All of the measures $p_y$ are $F$ invariant and ergodic.
\item
\label{sch:carried}
All of the measures $p_y$ satisfy $p_y(X \setminus \psi^{-1}(y)) = 0$.
\item
\label{sch:invar}
Let $\mathscr{Z}$ be the $\sigma$-algebra of $F$ invariant sets.
Let $\mathscr{C} = \{ \psi^{-1}(B) : B \in \mathscr{Y} \}$.
Then $\mathscr{Z}$ and $\mathscr{C}$ are $\nu$ equivalent.
\item
\label{sch:mapInf}
If there is another such collection $(Y',\mathscr{Y}',\psi',p')$ satisfying \ref{sch:mble} through \ref{sch:invar} then there is a measurable map $\Theta : Y \to Y'$ that is an isomorphism between $(Y,\mathscr{Y},\rho)$ and $(Y',\mathscr{Y}',\rho')$ such that $p'_y \sim p_{\Theta(y)}$ for $\rho$ almost every $y$.
\end{enumerate}
\end{corollary}
\begin{proof}
Let $\mu$ be  a probability measure that is equivalent to $\nu$.
Write $\mu = f \nu$ where $f$ is a positive, measurable function.
Theorem~\ref{thm:decompositionSchmidt} applied to $(X,\mathscr{B},F,\mu)$ gives $(Y,\mathscr{Y},\psi,q)$ satisfying \ref{sch:mble} through \ref{sch:invar} of Theorem~\ref{thm:decompositionSchmidt}.
Put $p_y = \frac 1 f q_y$.
It is straightforward to verify that $p$ is a disintegration of $\nu$ and therefore satisfies \ref{sch:disint} and \ref{sch:erg}.
It inherits the other properties from $q_y$. 
\end{proof}

Fix an interval exchange transformation $T$ and $f : [0,1) \to \mathbb{R}$ a mean-zero step function.
Take $X = [0,1) \times \mathbb{R}$ and let $\mathscr{B}$ be the Borel $\sigma$-algebra on $X$.
Apply Corollary~\ref{cor:decompositionSchmidt} with $F = T_f$ and $\nu = \haar$ to get $(Y,\mathscr{Y})$, the map $\psi$ and the family $y \mapsto p_y$ with the stated properties.
Write $\rho = \psi \haar$.
Let $\mathscr{Z}$ be the $\sigma$-algebra of $T_f$ invariant sets and let $\mathscr{C} = \{ \psi^{-1}(B) : B \in \mathscr{Y} \}$.
For $b \in \mathbb{R}$ define $V^b : X \to X$ by $V^b(x,t) = (x,t+b)$.
In preparation for the proof of Theorem~\ref{thm:gettingInvariance} we verify the following lemmas.

\begin{lemma}
\label{lem:transEquivar}
For every $b \in \mathbb{R}$ and $\haar$ almost every $(x,t)$ the measures $V^b p_{\psi(x,t)}$ and $p_{\psi(x,t+b)}$ are equivalent.
\end{lemma}
\begin{proof}
Define $p'_y = V^b p_y$ and $\psi'(x,t) = \psi(V^{-b}(x,t)) = \psi(x,t-b)$.
We claim that $(Y,\mathscr{Y},\psi',p')$ satisfies \ref{sch:mble} through \ref{sch:invar} of Theorem~\ref{cor:decompositionSchmidt}.
This is easily verified: we only check \ref{sch:disint} here by observing that $\psi' \haar = \psi \haar$ and calculating
\[
\iint 1_B \intd p'_y \intd \rho(y)
=
\iint 1_{V^{-b}B} \intd p_y \intd \rho(y)
=
\int 1_{V^{-b} B} \intd \haar
=
\int 1_B \intd \haar
\]
for all $B$ in $\mathscr{B}$.

We get from \ref{sch:mapInf} an automorphism $\Theta : Y \to Y$ such that $p'_{\Theta(y)} \sim p_y$ for $\rho$ almost every $y$.
Thus $p'_{\Theta(\psi(x,t))} \sim p_{\psi(x,t)}$ for $\haar$ almost every $(x,t)$.
So for $\haar$ almost every $(x,t)$ we have that the intersection $\psi^{-1}(\psi(x,t)) \cap (\psi')^{-1}(\Theta(\psi(x,t)))$ is co-null for both measures.
It follows that for $\haar$ almost every $(x,t)$ we have both $p_{\psi(z)} = p_{\psi(x,t)}$ and $p'_{\psi'(z)} = p'_{\Theta(\psi(x,t))}$ for $p_{\psi(x,t)}$ almost every $z$.
In conclusion $p'_{\psi'(z)} = p_{\psi(z)}$ for $\haar$ almost every $z \in [0,1) \times \mathbb{R}$.
In other words $V^b \psi_{\psi(x,t+b)} \sim p_{\psi(x,t)}$ for $\haar$ almost every $(x,t)$.
\end{proof}

Extending $T$ to $[0,1) \times \mathbb{R}$ by $T(x,t) = (Tx,t)$ we have the following lemma as well, whose proof is almost identical to that of the previous lemma.

\begin{lemma}
\label{lem:ietEquivar}
For $\haar$ almost every $(x,t)$ the measures $T p_{\psi(x,t)}$ and $p_{\psi(Tx,t)}$ are equivalent.
\end{lemma}

We also need the following lemmas.

\begin{lemma}
\label{lem:allSame}
Fix $b \in \mathbb{R}$.
If $p_y$ and $V^b p_y$ are not mutually singular for a set of positive $\rho$ measure then $p_y$ and $V^b p_y$ are not mutually singular for $\rho$ almost every $y$.
\end{lemma}
\begin{proof}
Put $H = \{(x,t) \in X : p_{\psi(x,t)} \not\perp V^b p_{\psi(x,t)} \}$.
We have $\haar(H) > 0$ by \ref{sch:carried} and our hypothesis.
It follows from Lemma~\ref{lem:ietEquivar} that $H$ is $T$ invariant.
Since $T$ is ergodic on $[0,1)$ the set $\{ x \in [0,1) : (x,t) \in H \}$ has either null of full Lebesgue measure for every $t \in \mathbb{R}$.
So our hypothesis gives a positive measure set of $t \in \mathbb{R}$ such that $\{ x \in [0,1) : (x,t) \in H \}$ has full Lebesgue measure.
But Lemma~\ref{lem:transEquivar} also implies that $H$ is $V^a$ invariant for all $a \in \mathbb{R}$.
So $H$ must be a co-null set.
\end{proof}

\begin{lemma}
If $\eta_y$ is the $\rho$ almost-surely defined normalized restriction of $p_y$ to $X_B$ then $y \mapsto \eta_y$ is an ergodic decomposition of $\haar_B$ for the transformation $S_{f,B}$.
\end{lemma}
\begin{proof}
Since almost every $p_y$ is ergodic for $T_f$ the restriction $\eta_y$ is almost surely ergodic for the induced transformation $S_{f,B}$.
That $y \mapsto p_y$ is a disintegration of $\haar_B$ follows from \ref{sch:disint}.
Thus $y \mapsto \eta_y$ is an ergodic decomposition of $\haar_B$.
\end{proof}

We are now ready to prove Theorem~\ref{thm:gettingInvariance}.

\begin{proof}
[Proof of Theorem~\ref{thm:gettingInvariance}]
Fix $b \in \mathbb{R}$ such that the set
\[
\{ (x,t) \in X_B : \mu_{f,B,(x,t)} \not\perp V^b \mu_{f,B,(x,t)} \}
\]
has positive $\haar_B$ measure.
Uniqueness in the ergodic decomposition~\cite[Theorem~6.2]{MR2723325} implies
\begin{equation}
\label{eqn:nonSingularSet}
\{ (x,t) \in X_B : p_{\psi(x,t)}|X_B \not\perp V^b \big( p_{\psi(x,t)}|X_B \big) \}
\end{equation}
has positive $\haar_B$ measure.
If $p_{\psi(x,t)}|X_B \not\perp V^b \big( p_{\psi(x,t)}|X_B \big)$ then we get from the Lebesgue decomposition theorem a measure $\lambda$ on $[0,1) \times \mathbb{R}$ which is absolutely continuous with respect to both measures.
So we have $\lambda \ll p_{\psi(x,t)}|X_B \ll p_{\psi(x,t)}$ and $\lambda \ll V^b (p_{\psi(x,t)}|X_B) \ll V^b p_{\psi(x,t)}$ whence $p_{\psi(x,t)}$ and $V^b p_{\psi(x,t)}$ are not mutually singular.
Since \eqref{eqn:nonSingularSet} has positive measure Lemma~\ref{lem:allSame} implies $p_{\psi(x,t)}\not \perp V^b p_{\psi(x,t)}$ for almost every $(x,t)$. 
Since ergodic measures are either mutually singular or equivalent we have 
\begin{equation}
\label{eq:equiv meas}
p_{\psi(x,t)} \sim V^bp_{\psi(x,t)}
\end{equation}
for almost every $(x,t)$. 

Our goal is to prove that $b$ is an essential value of $f$.
Fix $Z \in \mathscr{Z}_f$ a $T_f$ invariant set.
We wish to prove that $\haar(Z \triangle (V^b)^{-1} Z) = 0$.
By \ref{sch:invar} we may assume there exists $S \subset Y$ with $Z = \psi^{-1}(S)$.
For such $Z$ we have $p_y(Z) \in \{0,1\}$ for $\rho$ almost every $y$.
But then \eqref{eq:equiv meas} implies $p_y(Z) = p_y(V^{-b} Z)$ for $\rho$ almost every $y$.
Finally
\[
\haar( Z \triangle V^{-b} Z )
=
\int p_y (Z \triangle V^{-b} Z) \intd \rho(y)
=
0
\]
as desired.
\end{proof}

\section{Quantitative unique ergodicity}
\label{apdx:quantitativeUniqueErgodicity}

It follows from work of Boshernitzan~\cite[Theorem~1.7]{MR808101} that every linearly recurrent interval exchange transformation is uniquely ergodic.
In this section we prove the following quantitative version of Boshernitzan's result.
Throughout this section we use $b$ to denote the number of intervals of an interval exchange transformation.

\begin{theorem}
\label{thm:powerSaving}
Let $T$ be a linearly recurrent interval exchange transformation.
There is $0<\gamma<1$ so that for any mean-zero step function $f$ we have
\begin{equation}
\label{eqn:powerSaving}
\left| \sum_{n=0}^{N-1} f(T^n x) \right| \le N^\gamma
\end{equation}
for all large enough $N$. 
\end{theorem}

In fact, this result follows from Section~4 of \cite{chaikaConstantine}.
Most of this section constitutes a self-contained proof of Theorem~\ref{thm:powerSaving}, which we give for completeness.
Our interest in Theorem~\ref{thm:powerSaving} is in deducing from it that Property~\ref{fr:5} holds for every mean-zero step function $f : [0,1) \to \mathbb{R}$ and almost every $x$.
We begin with some notation for the induction scheme we will use throughout the proof of Theorem~\ref{thm:powerSaving}.

Fix a linearly recurrent interval exchange transformation $T$ with $b-1$ discontinuities.
Write $I_0$ for $[0,1)$ and let $I_{0,1},\dots,I_{0,b}$ be the intervals of continuity of $T$.
Define inductively $I_n = I_{n-1,1}$ and $I_{n,1},\dots,I_{n,b}$ as the intervals of the induced transformation $T|I_n$ on $I_n$. (See Figure~\ref{fig:inductionScheme} for a schematic.)
Since  we require defined linearly recurrent interval exchange transformations to satisfy the Keane condition, the induced transformation $T|I_n$ is also an exchange of $b$ intervals.
Given $\ell > k \ge 0$ define
\[
r_{k,\ell}(j) = \min \{ n \in \mathbb{N} : (T|I_k)^n I_{\ell,j} \subset I_\ell \}
\]
for all $1 \le j \le b$.
This is the first time the $T|I_k$ orbit of $I_{\ell,j}$ returns to $I_\ell$.
Write
\[
r_{k,\ell}(x) = \min \{ n \in \mathbb{N} : (T|I_k)^n x \in I_\ell \}
\]
for all $x \in I_k$.
Note that $r_{k,\ell}(j) = r_{k,\ell}(x)$ for all $x \in I_{\ell,j}$.
Define also for each $k \in \mathbb{N}$ a matrix $B_k$ with entries
\[
B_k(i,j) = \sum_{n=0}^{r_{k,k+1}(j)-1} 1_{I_{k,i}} \Big( (T|I_k)^n I_{k+1,j} \Big)
\]
for all $1 \le i,j \le b$ that count the number of visits of the $T|I_k$ orbit of $I_{k+1,j}$ to $I_{k,i}$ before the orbit visits $I_{k+1}$.
Therefore $B_k(1,j) = 1$ for all $1 \le j \le b$.
Note also that
\[
\nbar B_k e_j \nbar_1 = B_k(1,j) + \cdots + B_k(b,j) = r_{k,k+1}(j)
\]
for all $1 \le j \le b$ where $e_1,\dots,e_b$ is the standard basis of $\mathbb{R}^b$.

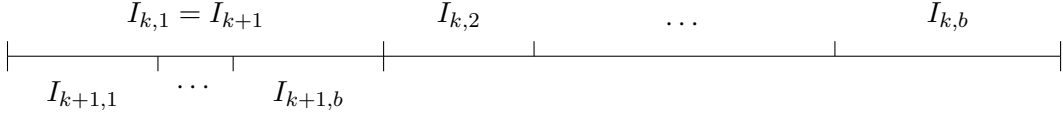
\begin{figure}[h]
\centering
\begin{tikzpicture}
\draw (-7,0) -- (7,0);
\draw (-7,-0.2) -- (-7,0);
\node [anchor=north] at (-6,-0.2) {$I_{k+1,1}$};
\draw (-5,-0.2) -- (-5,0);
\node [anchor=north] at (-4.5,-0.2) {$\cdots$};
\draw (-4,-0.2) -- (-4,0);
\node [anchor=north] at (-3,-0.2) {$I_{k+1,b}$};
\draw (-2,-0.2) -- (-2,0);
\draw (7,-0.2) -- (7,0);
\draw (-7,0) -- (-7,0.2);
\node [anchor=south] at (-4.5,0.2) {$I_{k,1} = I_{k+1}$};
\draw (-2,0) -- (-2,0.2);
\node [anchor=south] at (-1,0.2) {$I_{k,2}$};
\draw (0,0) -- (0,0.2);
\node [anchor=south] at (2,0.2) {$\cdots$};
\draw (4,0) -- (4,0.2);
\node [anchor=south] at (5.5,0.2) {$I_{k,b}$};
\draw (7,0) -- (7,0.2);
\end{tikzpicture}
\caption{The interval $I_k$ and some of its subintervals.}
\label{fig:inductionScheme}
\end{figure}
The entries of the matrix $B_{k,r} := B_k B_{k+1} \cdots B_{k+r}$ are
\[
B_{k,r}(i,j) = \sum_{n=0}^{r_{k,k+r+1}(j) - 1} 1_{I_{k,i}} \Big( (T|I_k)^n I_{k+r+1,j} \Big)
\]
and they count the number of visits of $I_{k+r+1,j}$ to $I_{k,i}$ under $T|I_k$ before it returns to $I_{k+r+1}$.
Therefore
\[
\nbar B_{k,r} e_j \nbar_1 = B_{k,r}(1,j) + \cdots + B_{k,r}(b,j) = r_{k,k+r+1}(j)
\]
for all $1 \le j \le b$.
Our proof of Theorem~\ref{thm:powerSaving} relies on the following facts.

\begin{fact}
\label{fact:returnRatio}
There is a constant $D_1 > 0$ such that $\frac{1}{D_1} < \frac{r_{k,l}(i)}{r_{k,l}(j)} < D_1$ for all $k > l\geq 0$ and all $1 \le i,j \le b$.
\end{fact}
\begin{proof}
Fix $x \in I_l$.
First note that
\begin{equation}
\label{eqn:returnBreakdown}
\min \{ r_{m,k}(z) : z \in I_m \}  r_{k,l}(x) \le r_{m,l}(x)  \le \max \{ r_{m,k}(z) : z \in I_m \} r_{k,l}(x)
\end{equation}
for all $l > k > m$ because each step in the $T|I_k$ orbit of $x$ involves a return of some point in $I_k$ to $I_k$ under $T|I_m$.
Now
\begin{equation}
\label{eqn:linRecIntervals}
c_2 \le r_{0,l}(x) |I_l| \le c_1
\end{equation}
for all $l$ by linear recurrence so
\begin{equation}
\label{eqn:cantThinkOfName}
\frac{c_2/|I_l|}{c_1/|I_k|} \le \frac{r_{0,l}(x)}{\max \{ r_{0,k}(z) : z \in I_0 \}} \le \frac{r_{0,l}(x)}{\min \{ r_{0,k}(z) : z \in I_0 \}} \le \frac{c_1/|I_l|}{c_2/|I_k|}
\end{equation}
for all $k$ and $l$.
Taking $m = 0$  in \eqref{eqn:returnBreakdown} and combining with \eqref{eqn:cantThinkOfName} shows that $D_1 = \dfrac{c_1^2}{c_2^2}$ works.
\end{proof}

\begin{fact}
\label{fact:intervalRatio}
There is a constant $D_2 > 0$ such that $\frac{1}{D_2} < \frac{|I_{k,j}|}{|I_{k,i}|} < D_2$ for all $1 \le i,j \le b$ and all $k \in \mathbb{N}$.
\end{fact}
\begin{proof}
Linear recurrence implies that the discontinuities of $T^{r_{0,k}(j)}$ are $c_1/r_{0,k}(j)$ dense and $c_2/r_{0,k}(j)$ separated.
Since $T|I_k$ is continuous on the interior of $I_{k,j}$ and has discontinuities at its endpoints we must have
\[
\frac{c_2}{\max \{ r_{0,k}(j) : 1 \le j \le b \} } \le |I_{k,j}| \le \frac{c_1}{\min \{ r_{0,k}(j) : 1 \le j \le b \} }
\]
so Fact~\ref{fact:returnRatio} implies  $D_2$ can be chosen to be $c_1 D_1/c_2$.
\end{proof}

\begin{fact}
\label{fact:geometricInterval}
There are constants $\rho_1, \rho_2 > 1$ such that $\rho_1 |I_{k,j}| \le |I_k| \le \rho_2 |I_{k,j}|$ for all $k \in \mathbb{N}$ and all $1 \le j \le b$.
In particular $\rho_1 |I_{k+1}| \le |I_k| \le \rho_2 |I_{k+1}|$ for all $k \in \mathbb{N}$.
\end{fact}
\begin{proof}
We have
\begin{equation}
\label{eqn:geometricIntervalLength}
|I_k|
=
|I_{k,1}| + |I_{k,2}| + \cdots + |I_{k,b}|
\le
|I_{k,j}| + (b-1) D_2 |I_{k,j}|
=
(1 + (b-1)D_2) |I_{k,j}|
\end{equation}
and
\[
|I_k|
=
|I_{k,1}| + |I_{k,2}| + \cdots + |I_{k,b}|
\ge
|I_{k,j}| + \frac{b-1}{D_2} |I_{k,j}|
=
\left( 1 + \frac{b-1}{D_2} \right) |I_{k,j}|
\]
by Fact~\ref{fact:intervalRatio}. Applying this when $j=1$ we can take $\rho_1 = 1 + (b-1)/D_2$ and $\rho_2 = 1 + (b-1)D_2$.
\end{proof}

\begin{fact}
\label{fact:normBound}
There is a constant $D_3 > 0$ such that $\nbar B_k \nbar_{1} \le D_3$ for all $k$.
\end{fact}
\begin{proof}
We have
\[
B_k(i,j)
\le
r_{k,k+1}(j)
\le
\frac{r_{0,k+1}(j)}{\min \{ r_{0,k}(l) : 1 \le l \le b \}}
\le
\frac{c_1}{c_2} \frac{|I_{k}|}{|I_{k+1}|}
\le
\frac{c_1}{c_2} \rho_2
\]
by taking $m = 0$ and $l = k+1$ in \eqref{eqn:returnBreakdown} and then using \eqref{eqn:cantThinkOfName} and Fact~\ref{fact:geometricInterval}.
It then follows that $\nbar B_k \nbar_{1} \le b c_1 \rho_2 / c_2$ for all $k$.
\end{proof}

\begin{fact}
\label{fac:productPositive}
There is $r \in \mathbb{N}$ such that $B_k B_{k+1} \cdots B_{k+r}$ is positive for all $k \in \mathbb{N}$.
\end{fact}
\begin{proof}
We must produce $r \in \mathbb{N}$ such that, for every $i,j,k$ the $T|I_k$ orbit of $I_{k+r+1,j}$ visits $I_{k,i}$ before returning to $I_{k+r+1}$.
It is enough to find $r$ such that, for every $i,j,k$ the $T$ orbit of any point $x \in I_{k+r+1,j}$ visits $I_{k,i}$ before time $r_{0,k+r+1}(j)$.
By linear recurrence this is the case if $r_{0,k+r+1}(j) |I_{k,i}| \ge c_1$.
Using Fact~\ref{fact:geometricInterval} repeatedly and then \eqref{eqn:linRecIntervals} we have
\[
r_{0,k+r+1}(j) |I_{k,i}|
\ge
r_{0,k+r+1}(j) \frac{|I_k|}{\rho_2}
\ge
r_{0,k+r+1}(j) |I_{k+r+1}| \frac{\rho_1^{r+1}}{\rho_2}
\ge
c_2 \frac{\rho_1^{r+1}}{\rho_2}
\]
and, independent of $i,j,k$, this will be at least $c_1$ if $r$ is large enough.
\end{proof}

\begin{fact}
\label{fac:angleDecay}
There are constants $D_4 \in \mathbb{R}$ and $\gamma < 1$ such that
\[
\Theta(B_{k,r} e_j, B_{k,r} e_\ell) < D_4 \gamma^r
\]
for all $j,\ell,k,r$ where $\Theta$ denotes the angle between two vectors in $\mathbb{R}^b$ and $e_1,\dots,e_b$ is the standard basis of $\mathbb{R}^b$.
\end{fact}
\begin{proof}
Because positive matrices of a fixed size act as definite contractions in the Hilbert projective metric, Fact~\ref{fac:productPositive} implies that there exists $\rho<1$ so that $\Theta(B_{k,k+r}v,B_{k,k+r}w)<\rho \Theta(v,w)$ for any $v,r \in \mathbb{R}_+^b$. Iterating we have $\Theta(B_{k,k+nr}v,B_{k,k+nr}w)<\rho^n$. 
Letting $\gamma=\rho^{\frac 1 r}$ and choosing $D_4=\rho^{-1}$ we obtain the fact.
\end{proof}

We use these facts to prove the following lemmas.

\begin{lemma}
There exists $0<\zeta_2$ and $E_2$ such that 
\[
\left| \sum_{n=0}^{N-1} 1_{I_k}(T^nx)- 1_{I_k}(T^n y) \right|
\leq
E_2\max\{1,(|I_k|N)^{\zeta_2}\}
\]
for all $x,y \in [0,1)$ and all $k,N \in \mathbb{N}$.
\end{lemma}
\begin{proof} 
Let $r=\frac 1 2 c\log(N|I_k|)$.  We first consider $x \in I_{k+r}$.
Let $M$ be the maximal number so that $M < N$ and $T^Mx=(T|I_{k+r})^ax \in I_{k+r}$.
Put
\[
\upsilon = \sum_{i=0}^M 1_{I_k} (T^ix) = |C_{i_1}(B_k \cdots B_{k+r})|+ \cdots +|C_{i_{a-1}}(B_k \cdots B_{k+r})|
\]
where $i_j$ is defined by $(T|I_{k+r})^jx\in I_{k+r,i_j}$.
By Fact~\ref{fac:angleDecay}
\[
\left| \frac{C_j(B_k \cdots B_{k+r})}{|C_j(B_k \cdots B_{k+r})|} - \frac{\upsilon}{M} \right|
\]
is exponentially small (in $\log(N|I_k|)$) for each $j$.
Now 
\[
\upsilon \leq \sum_{i=0}^{N-1} 1_{I_k}(T^ix)\leq |C_{\max}(B_k \cdots B_{k+r})| + \upsilon
\]
and by our choice of $r$ we have that $|C_{\max}(B_k \dots B_{k+r})|/M$ is exponentially small in $\log(N|I_k|)$ so
\[
\left| \frac 1 N\sum_{i=0}^{N-1} 1_{I_k}(T^ix)-\frac{C_j(B_k \cdots B_{k+r})}{|C_j(B_k \cdots B_{k+r})|} \right|
\]
is exponentially small, establishing the lemma for $x \in I_{k+r}$.
A general $x \in [0,1)$ gives another error of at most $|C_{\max}(B_k \cdots B_{k+r})|$.
\end{proof}

The proof of the next lemma is similar and omitted. 

\begin{lemma}
\label{lem:hitDifference3}
There is $0 < \zeta_3 < 1$ and $E_3 > 0$ such that
\[
\left| \sum_{n=0}^{N-1} 1_{I_{k,j}} (T^n x) - 1_{I_{k,j}} (T^n y) \right| < E_3 \max \{ 1, ( |I_{k,j}| N )^{\zeta_3} \}
\]
for all $k,j,N,x,y$.
\end{lemma}

\begin{corollary}\label{cor:hitDifference}
There is $0 < \zeta_3 < 1$ and $E_3 > 0$ such that
\[
\left| \sum_{n=0}^{N-1} 1_{T^s I_{k,j}} (T^n x) - 1_{T^s I_{k,j}} (T^n y) \right| < E_3 \max \{ 1, ( |I_{k,j}| N )^{\zeta_3} \}
\]
for all $k,j,N,x,y,s$.
\end{corollary}
\begin{proof}
This is immediate from Lemma~\ref{lem:hitDifference3} because $x$ and $y$ therein can be any point.
\end{proof}

With these lemmas we can prove Theorem~\ref{thm:powerSaving}.

\begin{proof}[Proof of Theorem~\ref{thm:powerSaving}]
Fix a mean-zero step function \eqref{eqn:simpleFunction}.
Set $a_0 = 0$ and $a_{d+1} = 1$.
Put $a_i = x_1 + \cdots + x_i$ for all $1 \le i \le d$.
For each $k \in \mathbb{N}$ the partition
\[
\mathcal{P}_k = \{ T^n I_{k,j} : 0 \le n < r_{0,k}(j), 1 \le j \le b \}
\]
of $[0,1)$ is $|I_k|$ dense.

Let $\zeta_3$ be as in Corollary~\ref{cor:hitDifference} and fix $\frac{1}{2} + \frac{\zeta_3}{2} < \gamma < 1$.
Fix $N \in \mathbb{N}$.
For each $1 \le i \le d+1$ choose $k$ minimal with $|I_k| \sqrt{N} < \sqrt{x_i}$.
Therefore $\sqrt{x_i} \le \rho_2 |I_k| \sqrt{N}$ by Fact~\ref{fact:geometricInterval}.
Write the interval $[a_{i-1},a_i)$ as a union of at most $x_i/|I_k|$ intervals from $\mathcal{P}_k$ together with an interval of length at most $|I_k|$ at each end.
Fix $x,y \in [0,1)$.
By estimating the hits of $x$ and $y$ to the end intervals using linear recurrence, and by comparing hits of $x$ and $y$ to the other intervals using Corollary~\ref{cor:hitDifference}, we obtain
\begin{align*}
\left| \sum_{n=0}^{N-1} 1_{[a_{i-1},a_i)} (T^n x) - \sum_{n=0}^{N-1} 1_{[a_{i-1},a_i)} (T^n y) \right|
&
\le
\frac{4}{c_2} \sqrt{N} \sqrt{x_i} + \frac{x_i}{|I_k|} E_3 \max \{1, (|I_k| N)^{\zeta_3} \}
\\
&
\le
\frac{4 \sqrt{N}}{c_2}
+
E_3 \rho_2 \sqrt{N}
\max \left\{ 1, N^{\frac{\zeta_3}2} \right\}
\end{align*}
for all $1 \le i \le d+1$ using $\sqrt{x_i} < 1$.
Combined with \eqref{eqn:simpleFunction} gives \eqref{eqn:powerSaving} by our choice of $\gamma$.
\end{proof}

\printbibliography

\end{document}